\documentclass[11pt, twoside, leqno]{article}

\usepackage{amsmath}
\usepackage{amssymb}
\usepackage{mathrsfs}
\usepackage{amsthm}
\usepackage{indentfirst}
\usepackage{color}
\usepackage{txfonts}

\allowdisplaybreaks

\def\bint{{\ifinner\rlap{\bf\kern.30em--}
\int\else\rlap{\bf\kern.35em--}\int\fi}\ignorespaces}

\def\sbint{{\ifinner\rlap{\bf\kern.32em--}
\hspace{0.078cm}\int\else\rlap{\bf\kern.45em--}\int\fi}\ignorespaces}

\def\red{\color{red}}

\def\rr{\mathbb{R}}
\def\rn{\mathbb{R}^n}

\def\cc{\mathbb{C}}
\def\nn{\mathbb{N}}
\def\zz{\mathbb{Z}}

\def\lz{\lambda}

\def\ls{\lesssim}

\def\fz{\infty}
\def\az{\alpha}

\def\cx{{\mathcal X}}

\def\BMO{\mathop\mathrm{\,BMO\,}}

\def\r{\right}
\def\lf{\left}

\def\noz{{\nonumber}}

\def\r{\right}
\def\lf{\left}

\def\at{{\mathop\mathrm{\,at\,}}}
\def\supp{{\mathop\mathrm{\,supp\,}}}

\def\loc{{\mathop\mathrm{\,loc\,}}}
\def\fin{{\mathop\mathrm{\,fin\,}}}

\def\BMO{{\mathop\mathrm{\,BMO\,}}}

\def\eqref#1{(\ref{#1})}

\newtheorem{theorem}{Theorem}[section]
\newtheorem{lemma}[theorem]{Lemma}
\newtheorem{corollary}[theorem]{Corollary}
\newtheorem{proposition}[theorem]{Proposition}

\theoremstyle{definition}
\newtheorem{remark}[theorem]{Remark}
\newtheorem{definition}[theorem]{Definition}

\numberwithin{equation}{section}

\begin{document}

\title{\bf\Large John--Nirenberg--Campanato Spaces
\footnotetext{\hspace{-0.35cm} 2010 {\it
Mathematics Subject Classification}. Primary 42B35; Secondary 42B30, 46E35. \endgraf
{\it Key words and phrases}. closed cube, Eudidean space,
John--Nirenberg--Campanato space, John--Nirenberg inequality, polymer, duality.
 \endgraf
This project is supported by the National
Natural Science Foundation of China   (Grant Nos. 11571039, 11761131002, 11671185 and 11871100).}}
\date{ }
\author{Jin Tao, Dachun Yang and Wen Yuan\footnote{Corresponding author/{\red  January 12, 2019}/Final version.}}
\maketitle

\vspace{-0.8cm}

\begin{center}
\begin{minipage}{13cm}
{\small {\bf Abstract}\quad
Let $p\in (1,\infty)$, $q\in[1,\infty)$, $\alpha\in [0,\infty)$
and $s$ be a non-negative integer. In this article, the authors introduce the John--Nirenberg--Campanato space
$JN_{(p,q,s)_\alpha}(\mathcal{X})$, where $\mathcal{X}$ is  $\rn$ or any closed cube
$Q_0\subset\rn$, which when $\alpha=0$ and $s=0$ coincides with the $JN_p$-space introduced
by F. John and L. Nirenberg in the sense of equivalent norms.
The authors then give the predual space of $JN_{(p,q,s)_\alpha}(\mathcal{X})$ and a
John--Nirenberg type inequality of John--Nirenberg--Campanato spaces.
Moreover, the authors prove that the classical Campanato space serves as a limit space of $JN_{(p,q,s)_\alpha}(\mathcal{X})$
when $p\to \infty$. }
\end{minipage}
\end{center}

\vspace{0.2cm}

\section{Introduction}

In their celebrated article \cite{JN61}, John and Nirenberg introduced the space $\BMO(\rn)$
which proves the dual space of the Hardy space $H^1(\rn)$ in \cite{FS72}.
A more general result was later given by  Coifman and Weiss \cite{CW77} who
proved that, for any $p\in(0,1]$, $q\in[1,\infty]$
and $s$ being the non-negative integer not smaller than $n(\frac1p-1)$,
the dual space of the Hardy space $H^p(\rn)$ is $\mathcal{C}_{\frac1p-1,\,q,\,s}(\rn)$,
the Campanato space introduced in \cite{Campanato64} which is known to  coincide with
$\BMO(\rn)$ when $p=1$.

In the same article \cite{JN61}, John and Nirenberg also introduced another
space $JN_p(Q_0)$ with $p\in(1,\infty)$ and showed that $JN_p(Q_0)\subset L^{p,\infty}(Q_0)$,
where $Q_0$ is a closed cube in $\rn$ and $Q_0\subsetneqq\rn$.
The limit of  $JN_p(Q_0)$ when $p\to\fz$ is just $\BMO(Q_0)$.
Since then, there have been several articles concerning the space $JN_p(Q_0)$
and some of their variants. For instance, Campanato \cite{Campanato66} used these spaces
to study the interpolation of operators;
Aalto et al. \cite{ABKY11} studied
$JN_p$ in the context of doubling metric spaces;
In \cite{hmv14}, Hurri-Syrj\"anen et al. established a local-to-global result for the
space $JN_p(\Omega)$ on an open subset $\Omega$ of $\rn$, namely, it was proved
that $\|\cdot\|_{JN_p(\Omega)}$ is dominated by its localized version
$\|\cdot\|_{JN_{p,\tau}(\Omega)}$ modulus constants,
where $\tau\in[1,\fz)$ and $\|\cdot\|_{JN_{p,\tau}(\Omega)}$ is defined as $\|\cdot\|_{JN_p(\Omega)}$
with an additional requirement $\tau Q\subset\Omega$ for all chosen cubes $Q$ in the definition
of $\|\cdot\|_{JN_p(\Omega)}$;
 Marola and   Saari \cite{ms16} studied the corresponding results on metric spaces
and obtained the equivalence between the local and the global $JN_p$ norm;
moreover, in both articles \cite{hmv14,ms16}, a global John--Nirenberg inequality for
$JN_p(\Omega)$ was established;
Berkovits et al. \cite{Berkovits16} applied the dyadic variant of $JN_p(Q_0)$
in the study of self-improving properties of some Poincar\'e type inequalities;
Dafni et al. \cite{DHKY18} showed the non-triviality of
$JN_p(Q_0)$ via proving $L^p(Q_0)\subsetneqq JN_p(Q_0)$ for any $p\in(1,\fz)$;
Dafni et al. \cite{DHKY18} also found the predual space of $JN_p(Q_0)$.
Very recently, A. Brudnyi and Y. Brudnyi \cite{BB18} introduced a more general function space
on the unit cube, which is a generalization of both classical BV spaces
and the John--Nirenberg spaces $JN_p([0,1]^n)$. It should be mentioned that various BMO type spaces have attracted lots of attentions in recent years; see, for instance, \cite{fpw98, mp98, BBM15,ABBF16,m16,ylk}.

Inspired by the relation between $\BMO$ and the Campanato spaces, as well as the relation between $\BMO$ and $JN_p$,
the purpose of this article is to introduce and study a Campanato type space $JN_{(p,q,s)_\alpha}(\mathcal{X})$
in spirit of the John--Nirenberg space
$JN_p(Q_0)$, which contains $JN_p(Q_0)$ as a special case. This John--Nirenberg--Campanato
space is defined not only on a closed cube $Q_0$ but
also on the whole space $\rn$.
The predual space
and a John--Nirenberg type inequality of John--Nirenberg--Campanato spaces are given.
It is also proved that the classical Campanato space serves as a limit space of $JN_{(p,q,s)_\alpha}(\mathcal{X})$.

Throughout the whole article, the \emph{symbol $\mathcal{X}$} represents either
$\rn$ or a closed cube $Q_0$ in $\rn$ and $Q_0\subsetneqq\rn$.
In what follows, we always use the \emph{symbol $Q_0$} to denote a closed cube of $\rn$
and $Q_0\subsetneqq\rn$, and
the \emph{symbol ${\mathbf 1}_E$} to denote the characteristic function of any set $E\subset\rn$.
For any integrable function $f$ and cube $Q$, let
$$f_Q:=\fint_Q f:=\frac{1}{|Q|}\int_Q f(x)\,dx.$$
Moreover, for any $s\in\zz_+$ (the set of all non-negative integers),
$P_Q^{(s)}(f)$  denotes the unique polynomial of degree not greater than $s$ such that
\begin{align}\label{P_Q^s(f)}
\int_Q \lf[f(x)-P_Q^{(s)}(f)(x)\r]x^\gamma\,dx=0,\quad \forall\,|\gamma|\le s,
\end{align}
where $\gamma:=(\gamma_1,\cdots,\gamma_n)\in \zz_+^n:=(\zz_+)^n$ and
$|\gamma|:=\gamma_1+\cdots+\gamma_n$.
It is well known that $P_Q^{(0)}(f)=f_Q$ and, for any $s\in\zz_+$,
there exists a constant $C_{(s)}\in[1,\infty)$, independent of $f$ and $Q$, such that
\begin{align}\label{C_{(s)}}
  \lf|P_Q^{(s)}(f)(x)\r|\le C_{(s)} \fint_Q |f|,\quad \forall\,x\in Q.
\end{align}
These fundamental properties can be found in, for instance,
\cite{Li08, LiangYang13, Lu95, TaiblesonWeiss80}.

Now, let us recall the definition of the Campanato space. In what follows,
for any $q\in[1,\infty)$, the \emph{symbol $L^q(\mathcal{X})$}
denotes the spaces of all measurable functions
$f$ such that $$\|f\|_{L^q(\mathcal{X})}:=\lf[\int_{\mathcal{X}}|f(x)|^q\,dx\r]^\frac1q<\infty$$
and the \emph{symbol $L_{\loc}^q(\mathcal{X})$} the set of all measurable functions $f$ such that
$f{\mathbf 1}_E\in L^q(\mathcal{X})$ for any bounded set $E\subset\mathcal{X}$.

\begin{definition}\label{Campanato}
  Let $\alpha\in[0,\infty)$, $q\in[1,\infty)$and $s\in\zz_+$.
  The \emph{Campanato space} $\mathcal{C}_{\alpha,q,s}(\mathcal{X})$
  is defined by setting
  $$\mathcal{C}_{\alpha,q,s}(\mathcal{X}):=\lf\{f\in L^1_{\loc}(\mathcal{X}):\,\,\|f\|_{\mathcal{C}_{\alpha,q,s}(\mathcal{X})}<\infty\r\},$$
  where $$\|f\|_{\mathcal{C}_{\alpha,q,s}(\mathcal{X})}:=
  \sup|Q|^{-\alpha}\lf[\fint_{Q}\lf|f-P_{Q}^{(s)}(f)\r|^q\r]
   ^{\frac 1q}$$
  and the supremum is taken over all closed cubes $Q$ in $\mathcal{X}$.

  The \emph{dual space $(\mathcal{C}_{\alpha,q,s}(\mathcal{X}))^\ast$} of $\mathcal{C}_{\alpha,q,s}(\mathcal{X})$
  is defined to be the set of all continuous linear functionals on $\mathcal{C}_{\alpha,q,s}(\mathcal{X})$
  equipped with weak-$\ast$ topology.
\end{definition}

Recall that, for any given $p\in(1,\infty)$, the \emph{John--Nirenberg space  $JN_p(Q_0)$}
is defined to be the set of all $f\in L^1(Q_0)$ such that
\begin{equation}\label{jnp}
\|f\|_{JN_p(Q_0)}:=\sup \lf[ \sum_i|Q_i|\lf(\fint_{Q_i}\lf|f-f_{Q_i}\r|
   \r)^p \r]^\frac1p<\infty,
\end{equation}
where $Q_0$ denotes a closed cube of $\rn$, $Q_0\subsetneqq\rn$ and
the supremum is taken over all collections of pairwise disjoint cubes
$\{Q_i\}_i$ in $Q_0$. It is obvious that $L^p(Q_0)\subset JN_p(Q_0)$,
while the nontriviality was showed by Dafni et al. \cite[Proposition 3.2]{DHKY18},
namely, $L^p(Q_0)\subsetneqq JN_p(Q_0)$. Dafni et al. \cite[Theorem 6.6]{DHKY18}
also found that the predual
of $JN_p(Q_0)$ is $HK_{p'}(Q_0)$, a Hardy-kind space, where $\frac1p+\frac1{p'}=1$.
Moreover, for any given $p\in(1,\infty)$ and $q\in [1,\fz)$, Dafni et al. \cite{DHKY18}
introduced the space $JN_{p,q}(Q_0)$ which is defined
to be the set of all $f\in L^1(Q_0)$ such that
$$\|f\|_{JN_{p,q}(Q_0)}:=\sup \lf[\sum_i|Q_i|\lf(\fint_{Q_i}\lf|f-f_{Q_i}\r|^q
   \r)^{\frac qp} \r]^\frac1p<\infty,$$
where the supremum is taken the same as in \eqref{jnp}. It was proved
in \cite[Proposition 5.1]{DHKY18} that, when $q\in [1,p)$,
$JN_{p,q}(Q_0)$ and $JN_p(Q_0)$ coincide with equivalent norms
and, when $q\in [p,\infty)$, $JN_{p,q}(Q_0)$ and $L^q(Q_0)$ as
sets coincide.

Motivated by the above two spaces, we introduce the John--Nirenberg--Campanato
space as follows, which is a natural generalization of the Campanato space;
see Proposition \ref{Campanato&JNpqas} below for more details.

\begin{definition}\label{JNpqa}
Let $p\in(1,\infty)$, $q\in[1,\infty)$, $s\in\zz_+$ and $\alpha\in[0,\infty)$.
The \emph{John--Nirenberg--Campanato space} $JN_{(p,q,s)_\alpha}(\mathcal{X})$
is defined by setting
$$JN_{(p,q,s)_\alpha}(\mathcal{X}):=\lf\{f\in L^1_{\loc}(\mathcal{X}):\,\,\|f\|_{JN_{(p,q,s)_\alpha}(\mathcal{X})}<\infty\r\},$$
where $$\|f\|_{JN_{(p,q,s)_\alpha}(\mathcal{X})}:=
\sup\lf\{\sum_i|Q_i|\lf[|Q_i|^{-\alpha}\lf\{\fint_{Q_i}\lf|f-P_{Q_i}^{(s)}(f)\r|^q\r\}
   ^{\frac 1q}\r]^p\r\}^{\frac 1p},$$
$P_{Q_i}^{(s)}(f)$ for any $i$ is as in \eqref{P_Q^s(f)} with $Q$ replaced by $Q_i$,
and the supremum is taken over all collections of pairwise disjoint cubes
$\{Q_i\}_i$ in $\mathcal{X}$.

The \emph{dual space $(JN_{(p,q,s)_\alpha}(\mathcal{X}))^\ast$} of $JN_{(p,q,s)_\alpha}(\mathcal{X})$
  is defined to be the set of all continuous linear functionals on $JN_{(p,q,s)_\alpha}(\mathcal{X})$
  equipped with weak-$\ast$ topology.
\end{definition}

\begin{remark}\label{relation}
Let $\mathcal{X}=Q_0$ and $\az=0=s$. In this case, it is obvious that $JN_{(p,q,s)_\alpha}(\mathcal{X})=JN_{p,q}(\mathcal{X})$ with equivalent norms.
Moreover, by \cite[Proposition 5.1]{DHKY18},
we know that, when $q\in [1,p)$, then $JN_{(p,q,s)_\alpha}(\mathcal{X})$
and $JN_p(\mathcal{X})$ coincide with equivalent norms and, when $q\in [p,\infty)$,
$JN_{(p,q,s)_\alpha}(\mathcal{X})$ and $L^q(\mathcal{X})$ coincide as sets.
\end{remark}

In what follows, for any $q\in(1,\fz)$ and any measurable function $f$, let
$$\|f\|_{L^q(Q_0,|Q_0|^{-1}dx)}:=\lf[\fint_{Q_0}|f(x)|^q\,dx\r]^\frac1q.$$
For any $q\in(1,\fz)$ and $s\in\zz_+$, the space
$L^q(Q_0,|Q_0|^{-1}dx)/\mathcal{P}_s(Q_0)$ is defined by setting
$$L^q(Q_0,|Q_0|^{-1}dx)/\mathcal{P}_s(Q_0)
:=\lf\{f\in L^q(Q_0):\,\,\|f\|_{L^q(Q_0,|Q_0|^{-1}dx)/\mathcal{P}_s(Q_0)}<\fz \r\},$$
where $\|f\|_{L^q(Q_0,|Q_0|^{-1}dx)/\mathcal{P}_s(Q_0)}
:=\inf_{m\in\mathcal{P}_s(Q_0)}\|f+m\|_{L^q(Q_0,|Q_0|^{-1}dx)}$
and $\mathcal{P}_s(Q_0)$ denotes the set of all polynomials of degree not greater
than $s$ on $Q_0$. In what follows, for a given positive constant $A$ and
a given function space $(\mathbb{X},\|\cdot\|_{\mathbb{X}})$,
we write $A\mathbb{X}:=\{Af:\,\,f\in\mathbb{X}\}$.
We have the following fine generalization of the fact $JN_{(p,q,0)_0}(Q_0)=L^q(Q_0)$
for any given $p\in(1,\fz)$ and $q\in[p,\fz)$ in \cite[Proposition 5.1]{DHKY18}.

\begin{proposition}\label{JNpq=Lq}
Let $p\in(1,\fz)$, $q\in[p,\fz)$, $s\in\zz_+$, $\alpha=0$ and a closed cube $Q_0\subsetneqq\rn$.
Then
$$|Q_0|^{-\frac1p} JN_{(p,q,s)_\alpha}(Q_0)
=L^q(Q_0,|Q_0|^{-1}dx)/\mathcal{P}_s(Q_0)$$
with equivalent norms.
\end{proposition}

We point out that it is a very interesting open question to find a counterpart of
Proposition \ref{JNpq=Lq} when $\alpha\in(0,\fz)$;
see Remark \ref{jnq1}(ii) below for more details.

The following proposition indicates that Campanato spaces server as the limit case of
John--Nirenberg--Campanato spaces as $p\to\infty$.
\begin{proposition}\label{Campanato&JNpqas}
  Let $\alpha\in[0,\infty)$, $q\in[1,\infty)$ and $s\in\zz_+$.
  Then, for any $f\in\cup_{r\in[1,\infty)}\cap_{p\in[r,\infty)}
  JN_{(p,q,s)_\alpha}(\mathcal{X})$, $f\in\mathcal{C}_{\alpha,q,s}(\mathcal{X})$ and
  $$\|f\|_{\mathcal{C}_{\alpha,q,s}(\mathcal{X})}=
  \lim_{p\to\infty}\|f\|_{JN_{(p,q,s)_\alpha}(\mathcal{X})}.$$
\end{proposition}

\begin{remark}\label{finite cube}
  Let $Q_0$ be a closed cube in $\rn$ and $Q_0\subsetneqq\rn$. It is well known that
  $$L^\infty(Q_0)=\lf\{ f\in\bigcap_{p\in[1,\infty)}L^p(Q_0):\,\,
      \lim_{p\to\infty}\|f\|_{L^p(Q_0)}<\infty\r\}$$
  and, for any $f\in L^\fz(Q_0)$,
  $$\|f\|_{L^\fz(Q_0)}=\lim_{p\to\fz}\|f\|_{L^p(Q_0)}.$$
  As was showed in \cite[p.\,7]{BB18}, we also have
  $$\BMO(Q_0)=\lf\{ f\in\bigcap_{p\in[1,\infty)}JN_p(Q_0):\,\,
      \lim_{p\to\infty}\|f\|_{JN_p(Q_0)}<\infty\r\}$$
  and, for any $f\in\BMO(Q_0)$,
  $$\|f\|_{\BMO(Q_0)}=\lim_{p\to\fz}\|f\|_{JN_p(Q_0)}.$$
\end{remark}

Proposition \ref{Campanato&JNpqas} then gives the following fine
generalization of the above results.

\begin{corollary}\label{p=8}
  Let $q\in[1,\infty)$, $\alpha\in[0,\infty)$, $s\in\zz_+$ and a closed cube $Q_0\subsetneqq\rn$.
  Then
  $$\mathcal{C}_{\alpha,q,s}(Q_0)=
    \lf\{ f\in\bigcap_{p\in(1,\infty)}JN_{(p,q,s)_\alpha}(Q_0):\,\,
      \lim_{p\to\infty}\|f\|_{JN_{(p,q,s)_\alpha}(Q_0)}<\infty\r\}$$
  and, for any $f\in \mathcal{C}_{\alpha,q,s}(Q_0)$,
  $$\|f\|_{\mathcal{C}_{\alpha,q,s}(Q_0)}=\lim_{p\to\fz}\|f\|_{JN_{(p,q,s)_\alpha}(Q_0)}.$$
\end{corollary}

Recall that the preduals of Campanato spaces are Hardy spaces \cite[p.\,55, Theorem 4.1]{Lu95}.
We show that the preduals of John--Nirenberg--Campanato spaces are some Hardy-type spaces.

For any $v\in[1,\infty]$, $s\in\zz_+$ and measurable subset $E\subset\rn$, let
\begin{align*}
L_s^v(E)&:=\lf\{f\in L^v(E):\,\,\ \supp(f)\subset E\mathrm{\,\,and\,\,}
  \int_E f(x)x^\gamma\,dx=0,\quad\forall\,|\gamma|\le s\r\},
\end{align*}
here and hereafter, $\supp(f):=\{x\in\rn:\,\,f(x)\neq0\}$.

\begin{definition}\label{uvas-atom}
Let $u\in(1,\infty)$, $v\in(1,\infty]$, $s\in\zz_+$ and $\alpha\in[0,\infty)$.
A function $a$ is called a \emph{$(u,v,s)_\alpha$-atom} on cube $Q$ if
\begin{enumerate}
  \item[{\rm(i)}] $\supp (a):=\{x\in\rn:\,\,a(x)\neq0\}\subset Q$;
  \item[{\rm(ii)}] $\|a\|_{L^v(Q)}\le |Q|^{\frac1v-\frac1u-\alpha}$,
    where, when $v=\fz$, $\|a\|_{L^v(Q)}$ represents the essential supremum of $a$ on $Q$;
  \item[{\rm(iii)}] $\int_Q a(x)x^\gamma\,dx=0$, $\,\,\forall\,\gamma\in\zz_+^n$
    and $|\gamma|\le s$.
\end{enumerate}
\end{definition}

\begin{remark}\label{classical atom}
\begin{itemize}
\item[(i)] Observe that, when $\az=0$, a $(u,v,s)_\alpha$-atom is just a classical \emph{$(u,v,s)$-atom}
(see, for instance, \cite{Stein93,Lu95,TaiblesonWeiss80}).

\item[(ii)] Let $p\in (0,1]$, $q\in [1,\fz]\cap(p,\fz]$ and $s\in\zz_+$. Following Coifman and Weiss \cite{CW77},
we define the \emph{atomic Hardy space $H^{p,q,s}_\at(\cx)$} to be the set of all
$f\in (\mathcal{C}_{\frac1p-1,q',s}(\mathcal{X}))^\ast$ when $p\in (0,1)$ or all $f\in L^1(\cx)$
when $p=1$ such that $f=\sum_{i\in\nn}\lz_ia_i$ in
$(\mathcal{C}_{\frac1p-1,q',s}(\mathcal{X}))^\ast$ when $p\in (0,1)$ or in $L^1(\cx)$
when $p=1$, where $\{a_i\}_{i\in\nn}$ are $(p,q,s)_0$-atoms and $\{\lz_i\}_{i\in\nn}\subset\cc$
satisfies $\sum_{i\in\nn}|\lz_i|^p<\fz$. Moreover, for any $f\in H^{p,q,s}_\at(\cx)$, let
$$\|f\|_{H^{p,q,s}_\at(\cx)}:=\inf\lf(\sum_{i\in\nn}|\lz_i|^p\r)^{\frac1p},$$
where the infimum is taking over all decompositions of $f$ as above.
\end{itemize}
\end{remark}

\begin{definition}\label{uvas-polymer}
Let $u\in(1,\infty)$, $v\in(1,\infty]$, $s\in\zz_+$ and $\alpha\in[0,\infty)$.
The \emph{space of $(u,v,s)_\alpha$-polymers}, $\widetilde{HK}_{(u,v,s)_\alpha}(\mathcal{X})$,
is defined by setting
\begin{align*}
\widetilde{HK}_{(u,v,s)_\alpha}(\mathcal{X})
:=&\lf\{g:\,\,
g=\sum_j \lambda_j a_j \,\,\mathrm{pointwise},\,\,
\{a_j\}_j \,\,\mathrm{are}\,\,(u,v,s)_\alpha\mathrm{\emph{-}atoms}\,\,\mathrm{on}\r.\\
&\quad\quad\quad\lf.\,\,\mathrm{disjoint\,\,cubes}\,\,\{Q_j\}_j\subset\mathcal{X}\,\,
\mathrm{and\,\,}\sum_j \lf|\lambda_j\r|^u<\infty\r\}
\end{align*}
and, for any $g\in \widetilde{HK}_{(u,v,s)_\alpha}(\mathcal{X})$, let
$$\|g\|_{\widetilde{HK}_{(u,v,s)_\alpha}(\mathcal{X})}
:=\inf\lf(\sum_j \lf|\lambda_j\r|^u\r)^\frac1u,$$
where the infimum is taken over all decompositions of $g$ as above.

Moreover, any $g\in \widetilde{HK}_{(u,v,s)_\alpha}(\mathcal{X})$ is called a
\emph{$(u,v,s)_\alpha$-polymer}.
\end{definition}

\begin{remark}\label{polymerJFA}
\begin{itemize}
\item[(i)] Recall that, when $\mathcal{X}=Q_0$, for any $1<u<v\le\infty$,
Dafni et al. \cite[Definition 6.1]{DHKY18} introduced
$(u,v)$-polymers as follows.
A function $g$ is called a $(u,v)$-\emph{polymer} if $g=\sum_j a_j$ pointwise,
where $\{Q_j\}_j$ is a collection of disjoint cubes in $Q_0$ and, for any $j$,
$a_j\in L_0^v(Q_j)$ and
$$\lf\|\lf\{a_j\r\}_j\r\|_{(u,v)}:=
\lf[\sum_j\lf|Q_j\r|\lf(\fint_{Q_j}\lf|a_j\r|^v\r)^\frac uv\r]^\frac 1u<\infty,$$
with the usual modification made when $v=\infty$. The norm $\|g\|_{P(u,v)}$
is then defined to be the infimum of $\|\{a_j\}_j\|_{(u,v)}$ over all such
representations of $g$ as above.

We now claim that, when $\alpha=0=s$, $(u,v)$-polymer and $(u,v,s)_\alpha$-polymer coincide with equivalent norms. Indeed, for any $(u,v)$-polymer $g$ and $\epsilon\in(0,\fz)$, there exist
$\{a_j\}_j$ as above such that $g=\sum_j a_j$ and
$$\lf\|\lf\{a_j\r\}_j\r\|_{(u,v)}=
\lf[\sum_j\lf|Q_j\r|\lf(\fint_{Q_j}\lf|a_j\r|^v\r)^\frac uv\r]^\frac 1u
\le \|g\|_{P(u,v)}+\epsilon .$$
For any $j$, normalize $a_j$ by setting
$a_j:=\lambda_j\widetilde{a_j}$
with $\lambda_j:=\|a_j\|_{L^v(Q_j)}|Q_j|^{\frac1u-\frac1v}$.
Then $g=\sum_j a_j=\sum_j \lambda_j \widetilde{a_j}$,
$\{\widetilde{a_j}\}_j$ are $(u,v,0)_0$-atoms, and
$$\lf\|\lf\{a_j\r\}_j\r\|_{(u,v)}
=\lf[\sum_j\lf|Q_j\r|\lf(\lf|Q_j\r|^{-\frac1v}\lambda_j\lf|Q_j\r|
    ^{\frac1v-\frac1u}\r)^u\r]^\frac 1u
=\lf( \sum_j\lambda_j^u\r)^\frac1u.$$
This implies that
$$\|g\|_{\widetilde{HK}_{(u,v,0)_0}(Q_0)}
\le\lf( \sum_j\lambda_j^u\r)^\frac1u\le\|g\|_{P(u,v)}+\epsilon$$
and hence $\|g\|_{\widetilde{HK}_{(u,v,0)_0}(Q_0)}\le\|g\|_{P(u,v)}$
by the arbitrariness of $\epsilon\in(0,\fz)$.
On the other hand, for any $(u,v,0)_0$-polymer $g$ and $\epsilon\in(0,\fz)$,
there exist $(u,v,0)_0$-atoms $\{a_j\}_j$ and $\{\lambda_j\}_j\subset\cc$ such that
$g=\sum_j \lambda_j a_j$ and
$$\lf(\sum_j \lf|\lambda_j\r|^u\r)^\frac1u
\le \|g\|_{\widetilde{HK}_{(u,v,0)_0}(Q_0)}+\epsilon.$$
Thus,
\begin{align*}
\|g\|_{P(u,v)}
&\le\lf\|\lf\{\lambda_j a_j\r\}_j\r\|_{(u,v)}
 =\lf[\sum_j\lf|Q_j\r|\lf(\fint_{Q_j}\lf|\lambda_j a_j\r|^v\r)^\frac uv\r]^\frac 1u
 =\lf[\sum_j \lf|\lambda_j\r|^u \lf|Q_j\r|^{1-\frac uv}\lf\|a_j\r\|_{L^v(Q_j)}^u\r]^\frac 1u\\
&\le \lf[\sum_j \lf|\lambda_j\r|^u \lf|Q_j\r|^{1-\frac uv}\lf|Q_j\r|^{(\frac1v-\frac1u)u}\r]^\frac 1u
 =\lf( \sum_j\lf|\lambda_j\r|^u\r)^\frac1u\le\|g\|_{\widetilde{HK}_{(u,v,0)_0}(Q_0)}+\epsilon
\end{align*}
and hence, by the arbitrariness of $\epsilon\in(0,\fz)$ again,
$\|g\|_{P(u,v)}\le\|g\|_{\widetilde{HK}_{(u,v,0)_0}(Q_0)}$.
Therefore, $\|g\|_{\widetilde{HK}_{(u,v,0)_0}(Q_0)}=\|g\|_{P(u,v)}$,
which shows the above claim.

\item[(ii)]Comparing with Dafni et al. \cite[Definition 6.1]{DHKY18}, we prefer
to use $(u,v,s)_\alpha$-atoms to define $(u,v,s)_\alpha$-polymers because,
in this way, it is easy to see the difference and the commonality between
$\widetilde{HK}_{(u,v,s)_\alpha}(\mathcal{X})$ and the classical atomic Hardy space;
see Remarks \ref{classical atom} and \ref{jnq1}(iii), and
also Dafni et al. \cite[Remark 6.3]{DHKY18}.
\end{itemize}
\end{remark}

In what follows, for any $u\in[1,\infty]$, let $u'$ denote its \emph{conjugate index},
namely, $\frac1u+\frac1{u'}=1$.
Proposition \ref{weak-star} below leads us to  define the space
$HK_{(u,v,s)_\alpha}(\mathcal{X})$ via the weak-$\ast$ topology as in \cite{CW77}.

\begin{proposition}\label{weak-star}
  Let $u\in(1,\infty)$, $v\in (1,\infty]$, $s\in\zz_+$ and $\alpha\in[0,\fz)$.
  Then $\widetilde{HK}_{(u,v,s)_\alpha}(\mathcal{X})\subset
  (JN_{(u',v',s)_\alpha}(\mathcal{X}))^\ast$.
\end{proposition}

\begin{definition}\label{HKuvb}
Let $u\in(1,\infty)$, $v\in (1,\infty]$, $s\in\zz_+$ and $\alpha\in[0,\fz)$.
The \emph{Hardy-type space} $HK_{(u,v,s)_\alpha}(\mathcal{X})$ is defined by setting
\begin{align*}
HK_{(u,v,s)_\alpha}(\mathcal{X}):=&\lf\{g\in(JN_{(u',v',s)_\alpha}(\mathcal{X}))^\ast:\,\,
g=\sum_i g_i \,\,\mathrm{in}\,\, (JN_{(u',v',s)_\alpha}(\mathcal{X}))^\ast,\,\,
\{g_i\}_i\subset\widetilde{HK}_{(u,v,s)_\alpha}(\mathcal{X})\r.\\
&\quad\quad\quad\lf.\mathrm{and\,\,}
\sum_i\lf\|g_i\r\|_{\widetilde{HK}_{(u,v,s)_\alpha}(\mathcal{X})}<\infty\r\}
\end{align*}
and, for any $g\in HK_{(u,v,s)_\alpha}(\mathcal{X})$, let
$$\|g\|_{HK_{(u,v,s)_\alpha}(\mathcal{X})}:=\inf\sum_i
\|g_i\|_{\widetilde{HK}_{(u,v,s)_\alpha}(\mathcal{X})},$$
where the infimum is taken over all decompositions of $g$ as above.

Moreover, the \emph{finite atomic Hardy-type space
$HK_{(u,v,s)_\alpha}^{\mathrm{fin}}(\mathcal{X})$} is defined to be
the set of all finite sum $\sum_{m=1}^M \lambda_{m}a_{m}$,
where $M\in\nn$, $\{\lambda_{m}\}_{m=1}^M\subset\cc$ and
$\{a_{m}\}_{m=1}^M$ are $(u,v,s)_\alpha$-atoms.
\end{definition}

Following Dafni et al. \cite[Definition 6.1]{DHKY18}, in Definition \ref{HKuvb},
we use the symbol $HK_{(u,v,s)_\alpha}(\mathcal{X})$ to denote these Hardy-type
spaces, where $HK$ in \cite{DHKY18} probably means Hardy-kind.
Recall that, for any $u\in(1,\fz)$ and $v\in(u,\fz]$,
Dafni et al. \cite[Definition 6.1]{DHKY18}
introduced the \emph{Hardy-kind space} $HK_{u,v}(Q_0)$ which is defined
to be the space of all $g\in L^u(Q_0)$ such that
$g=\sum_i g_i$, where each $g_i$ is a $(u,v)$-polymer as in Remark \ref{polymerJFA},
and $\sum_i \|g_i\|_{P(u,v)}<\fz$. Besides, they also define
$\|g\|_{HK_{u,v}(Q_0)}$ to be the infimum of
$\sum_i\|g_i\|_{P(u,v)}$ over all such decompositions.
Indeed, $HK_{(u,v,0)_0}(\mathcal{X})$ and $HK_{u,v}(\mathcal{X})$
coincide with equivalent norms, which is a direct consequence of the following proposition.

\begin{proposition}\label{case0}
Let $Q_0$ be a closed cube in $\rn$, $Q_0\subsetneqq\rn$, $u\in(1,\fz)$ and $v\in(u,\fz]$.
If $g=\sum_{i\in\nn} g_i$ in $(JN_{(u',v',0)_0}(Q_0))^\ast$, where
$\{g_i\}_{i\in\nn}\subset\widetilde{HK}_{(u,v,0)_0}(Q_0)$
and $\sum_{i\in\nn} \|g_i\|_{\widetilde{HK}_{(u,v,0)_0}(Q_0)}<\fz$,
then there exists some $\widetilde{g}\in L^u(Q_0)$ such that
$\widetilde{g}=\sum_i g_i$ in $L^u(Q_0)$ and
$g=\widetilde{g}$ in $(JN_{(u',v',0)_0}(Q_0))^\ast$, namely,
for any $f\in JN_{(u',v',0)_0}(Q_0)$,
$$\langle \widetilde{g},f \rangle=\langle g,f \rangle
=\sum_i \langle g_i,f \rangle.$$
\end{proposition}

\begin{remark}\label{predual}
By Proposition \ref{case0}, we know that, for any $u\in(1,\fz)$ and $v\in(u,\fz]$, the spaces
$HK_{(u,v,0)_0}(Q_0)$ and $HK_{u,v}(Q_0)$ coincide with equivalent norms.
\end{remark}

The duality relation reads as follows.

\begin{theorem}\label{duality}
Let $p\in(1,\infty)$, $q\in [1,\fz)$, $s\in\zz_+$ and $\alpha\in[0,\fz)$.
Then $(HK_{(p',q',s)_\alpha}(\mathcal{X}))^\ast=JN_{(p,q,s)_\alpha}(\mathcal{X})$
in the following sense:
\begin{enumerate}
  \item[{\rm(i)}]If $f\in JN_{(p,q,s)_\alpha}(\mathcal{X})$, then $f$ induces a linear
  functional $\mathcal{L}_f$ on $HK_{(p',q',s)_\alpha}(\mathcal{X})$ with
  $$\|\mathcal{L}_f\|_{(HK_{(p',q',s)_\alpha}(\mathcal{X}))^\ast}
    \le C\|f\|_{JN_{(p,q,s)_\alpha}(\mathcal{X})},$$
  where $C$ is a positive constant independent of $f$.
  \item[{\rm(ii)}] If $\mathcal{L}\in(HK_{(p',q',s)_\alpha}(\mathcal{X}))^\ast$,
  then there exists some $f\in JN_{(p,q,s)_\alpha}(\mathcal{X})$ such that
  $$\mathcal{L}(g)=\int_{\mathcal{X}}f(x)g(x)\,dx,\quad\forall\,g\in
   HK_{(p',q',s)_\alpha}^{\mathrm{fin}}(\mathcal{X})$$
  and
  $$\|\mathcal{L}\|_{(HK_{(p',q',s)_\alpha}(\mathcal{X}))^\ast}
  \sim\|f\|_{JN_{(p,q,s)_\alpha}(\mathcal{X})}$$
  with the positive equivalence constants independent of $f$.
\end{enumerate}
\end{theorem}

\begin{remark}\label{dual}
When $\mathcal{X}=Q_0$, $\az=0=s$ and $q\in[1,p)$, by Proposition \ref{case0},
we know that Theorem \ref{duality} coincides with \cite[Theorem 6.6]{DHKY18}.
\end{remark}

Combining Theorem \ref{duality} and Proposition \ref{JNpq=Lq}, we have the following
atomic characterization of $L^{q'}_s(Q_0)$ for any given $q'\in(1,\fz)$ and $s\in\zz_+$.

\begin{corollary}\label{HKqq=Lq}
Let $p\in(1,\fz)$, $q\in[p,\fz)$, $s\in\zz_+$, $\alpha=0$ and a closed cube $Q_0\subsetneqq\rn$.
Then
$$L^{q'}_s(Q_0,|Q_0|^{q'-1}dx)=|Q_0|^\frac1p HK_{(p',q',s)_\alpha}(Q_0)$$
with equivalent norms.
\end{corollary}

Recall that, if $p\in (0,1]$, the classical atomic Hardy spaces $H_{\mathrm{at}}^{p,q,0}(\rn)$
are invariant on $q\in[1,\infty]\cap(p,\fz]$ (see, for instance,
\cite[Theorem A]{CW77}).
A similar phenomenon also appears in the Hardy-kind space $HK_{p,q}(Q_0)$ introduced
in \cite{DHKY18}, which was the predual space
of $JN_{p'}(Q_0)$ (see \cite[Theorem 6.6]{DHKY18}).
In Proposition \ref{HKuvb=HKu8b} below, we show that  the space
$HK_{(u,v,s)_\alpha}(\mathcal{X})$ also keeps such a property.
To this end, we need the following property of $JN_{(p,q,s)_\alpha}(\mathcal{X})$.

\begin{proposition}\label{JNpqa=JNp1a}
  Let $1\le q<p<\infty$, $s\in\zz_+$ and $\alpha\in[0,\infty)$.
  Then $JN_{(p,q,s)_\alpha}(\mathcal{X})=JN_{(p,1,s)_\alpha}(\mathcal{X})$
  with equivalent norms.
\end{proposition}

\begin{remark}\label{jnq1}
\begin{itemize}
\item[(i)] When $X=Q_0$, $\az=0=s$ and $q=1$, Proposition \ref{JNpqa=JNp1a}
was obtained in \cite[Proposition 5.1]{DHKY18}.
\item[(ii)] From Proposition \ref{JNpq=Lq}, we deduce that $JN_{(p,q,s)_0}(Q_0)=JN_{(q,q,s)_0}(Q_0)$
for any given $p\in(1,\fz)$, $q\in [p,\fz)$ and $s\in\zz_+$.
By this and Proposition \ref{JNpqa=JNp1a}, we conclude that,
for any $p\in(1,\fz)$ and $s\in\zz_+$,
$$JN_{(p,q,s)_0}(Q_0)=
\begin{cases}
JN_{(p,1,s)_0}(Q_0),& q\in [1,p),\\
JN_{(q,q,s)_0}(Q_0),& q\in [p,\fz).
\end{cases}$$

We further claim that, for any given $p\in(1,\fz)$, $q\in [p,\fz)$, $s\in\zz_+$,
$\alpha\in [0,\infty)$ and a closed cube $Q_0\subsetneqq\rn$,
$JN_{(q,q,s)_\az}(Q_0)\subset JN_{(p,q,s)_\az}(Q_0)$ and,
for any $f\in JN_{(q,q,s)_\az}(Q_0)$,
\begin{equation}\label{JNpqJNqq}
|Q_0|^{-\frac1p}\|f\|_{JN_{(p,q,s)_\az}(Q_0)}
\le |Q_0|^{-\frac1q}\|f\|_{JN_{(q,q,s)_\az}(Q_0)}.
\end{equation}
Indeed, for the given $f\in JN_{(q,q,s)_\az}(Q_0)$, by Definition \ref{JNpqa}, we
know that $f\in L^1(Q_0)$. Thus, for any disjoint cubes $\{Q_i\}_i$ in $Q_0$,
by the Jensen inequality, we conclude that
\begin{align*}
  &\lf\{\sum_i\frac{|Q_i|}{|Q_0|}\lf[|Q_i|^{-\az}\lf\{\fint_{Q_i}\lf|
   f-P^{(s)}_{Q_i}(f) \r|^q \r\}^\frac1q \r]^p \r\}^\frac1p\\
  &\quad=\lf\{\sum_i\frac{|Q_i|}{|Q_0|}\lf[|Q_i|^{-\az q}\fint_{Q_i}\lf|
   f-P^{(s)}_{Q_i}(f) \r|^q \r]^\frac pq \r\}^\frac1p\\
  &\quad\le\lf\{\lf[\sum_i\frac{|Q_i|}{|Q_0|}|Q_i|^{-\az q}\fint_{Q_i}\lf|
   f-P^{(s)}_{Q_i}(f) \r|^q \r]^\frac pq \r\}^\frac1p\\
  &\quad=\lf[\sum_i\frac{|Q_i|}{|Q_0|}|Q_i|^{-\az q}\fint_{Q_i}\lf|
   f-P^{(s)}_{Q_i}(f) \r|^q \r]^\frac 1q\\
  &\quad=\lf[\sum_i\frac{|Q_i|}{|Q_0|}\lf\{|Q_i|^{-\az }\lf[\fint_{Q_i}\lf|
   f-P^{(s)}_{Q_i}(f) \r|^q \r]^\frac 1q \r\}^q \r]^\frac1q,
\end{align*}
which further implies that \eqref{JNpqJNqq} and
hence $JN_{(q,q,s)_\az}(Q_0)\subset JN_{(p,q,s)_\az}(Q_0)$ hold true.

Besides, for any $p\in(1,\fz)$, $q\in [p,\fz)$, $s\in\zz_+$ and
$\alpha\in (\frac{s+1}n,\infty)$, by \cite[Lemma 3.1]{BB18}, we know that
$JN_{(p,q,s)_\alpha}([0,1]^n)=\mathcal{P}_s([0,1]^n)$.
It is easy to see that this is also true when $[0,1]^n$ is replaced by
any closed cube $Q_0\subsetneqq\rn$.
However, it is still unknown whether or not
$JN_{(p,q,s)_\alpha}(Q_0)=JN_{(q,q,s)_\alpha}(Q_0)$ still holds true for any
$p\in(1,\fz)$, $q\in [p,\fz)$, $s\in\zz_+$ and $\alpha\in(0,\frac{s+1}n]$.

Moreover, it is also still unknown whether or not
$JN_{(p,q,s)_\alpha}(\rn)=JN_{(q,q,s)_\alpha}(\rn)$ still holds true for any
$p\in(1,\fz)$, $q\in [p,\fz)$, $s\in\zz_+$ and $\alpha\in[0,\fz)$.
\item[(iii)] Let $u\in(1,2)$, $v\in(1,\fz]$, $s\in\zz_+$, $\alpha\in[0,\fz)$ and
a closed cube $Q_0\subsetneqq\rn$.
We claim that, for any $u\in(1,2)$,
$HK_{(u,v,s)_\az}(Q_0)\subset H^{\frac1{\az+1},v,s}_{\mathrm{at}}(Q_0)$ and,
for any $g\in\bigcap_{u\in(1,2)}HK_{(u,v,s)_\az}(Q_0)$,
\begin{equation}\label{limHKu}
  \|g\|_{H^{\frac1{\az+1},v,s}_{\mathrm{at}}(Q_0)}
  \le \liminf_{u\to1^+} \|g\|_{HK_{(u,v,s)_\az}(Q_0)},
\end{equation}
here and hereafter, $\liminf_{u\to1^+}$ means that $u\in (1,2)$ and $u\to 1$.

Indeed, for any given $g\in HK_{(u,v,s)_\az}(Q_0)$, by Definition \ref{HKuvb},
we know that there exist $(u,v,s)_\az$-polymers $\{g_i\}_i$ such that
$g=\sum_i g_i$ in $(JN_{(u',v',s)_\az}(Q_0))^\ast$ and
\begin{equation}\label{uto1.1}
\sum_i\|g_i\|_{\widetilde{HK}_{(u,v,s)_\az}(Q_0)}\ls \|g\|_{HK_{(u,v,s)_\az}(Q_0)}.
\end{equation}
When $\az=0$, by Proposition \ref{case0}, without loss of generality,
we may assume that $g=\sum_i g_i$ in $L^u(Q_0)$.
From this and the fact that $L^u(Q_0)\subset L^1(Q_0)$,
we further deduce that  $g=\sum_i g_i$ in $L^1(Q_0)$ as well.
When $\az\in(0,\fz)$,
from this and the fact that $\mathcal{C}_{\az,v',s}(Q_0)\subset JN_{(u',v',s)_\az}(Q_0)$
which is a simple consequence of Corollary \ref{p=8},
we further deduce that  $g=\sum_i g_i$ in $(\mathcal{C}_{\az,v',s}(Q_0))^\ast$ as well.

Moreover, for any $\az\in[0,\fz)$ and any $i$, by Definition \ref{uvas-polymer}, we know that there exist
$\{\lambda_{i,j}\}_j\subset\cc$ and $(u,v,s)_\az$-atoms $\{g_{i,j}\}_j$ on disjoint cubes
$\{Q_{i,j}\}_j\subset Q_0$ such that $g_i=\sum_j g_{i,j}$ pointwise and
\begin{equation}\label{uto1.2}
\lf(\sum_j\lf|\lambda_{i,j}\r|^u \r)^\frac1u\ls\|g_i\|_{\widetilde{HK}_{(u,v,s)_\az}(Q_0)}.
\end{equation}
For any $i,\,j$, let $\widetilde{g_{i,j}}:=|Q_{i,j}|^{1/u-1}g_{i,j}$.
Then $\{\widetilde{g_{i,j}}\}_{i,j}$ are classical $(\frac1{\az+1},v,s)$-atoms
(see, for instance, Remark \ref{classical atom}) and
$g=\sum_{i,j}|Q_{i,j}|^{1-1/u}\lambda_{i,j}\widetilde{g_{i,j}}$
in $L^1(Q_0)$ when $\az=0$ or in $(\mathcal{C}_{\az,v',s}(Q_0))^\ast$
when $\az\in (0,\fz)$.
From this, the H\"{o}lder inequality, the disjointness of $\{Q_{i,j}\}_j$ for any $i$,
\eqref{uto1.1} and \eqref{uto1.2}, we deduce that
\begin{align*}
  \|g\|_{H^{\frac1{\az+1},v,s}_{\mathrm{at}}(Q_0)}
  &\le \sum_i \sum_j \lf|Q_{i,j}\r|^{1-\frac1u}\lf|\lambda_{i,j}\r|
  \le \sum_i \lf[\sum_j\lf|Q_{i,j}\r|^{(1-\frac1u)u'}\r]^\frac1{u'}
             \lf(\sum_j\lf|\lambda_{i,j}\r|^u \r)^\frac1u\\
  &\le |Q_0|^\frac1{u'} \sum_i \lf(\sum_j\lf|\lambda_{i,j}\r|^u \r)^\frac1u
  \ls |Q_0|^\frac1{u'} \sum_i \|g_i\|_{\widetilde{HK}_{(u,v,s)_\az}(Q_0)}\\
  &\ls |Q_0|^\frac1{u'} \|g\|_{HK_{(u,v,s)_\az}(Q_0)},
\end{align*}
which further implies that, for any given $u\in(1,2)$, $v\in(1,\fz]$, $s\in\zz_+$
and $\az\in[0,\fz)$,
$HK_{(u,v,s)_\az}(Q_0)\subset H^{\frac1{\az+1},v,s}_{\mathrm{at}}(Q_0)$
and \eqref{limHKu} holds true.
This finishes the proof of the above claim.

Moreover, for any given $v\in(1,\fz]$, $s\in\zz_+$, $\alpha\in[0,\fz)$ and
a closed cube $Q_0\subsetneqq\rn$, as a counterpart of Proposition \ref{Campanato&JNpqas},
it is very interesting to ask whether or not, for any $g\in\bigcap_{u\in(1,2)}HK_{(u,v,s)_\az}(Q_0)$
$$\|g\|_{H^{\frac1{\az+1},v,s}_{\mathrm{at}}(Q_0)}
=\lim_{u\to1^+} \|g\|_{HK_{(u,v,s)_\az}(Q_0)}$$
holds true, which is still unknown.
\end{itemize}
\end{remark}

Proposition \ref{JNpqa=JNp1a} is proved via the following
John--Nirenberg type inequality, which is also of independent interest.
In what follows, for any given $p\in[1,\fz)$, the
\emph{weak Lebesgue space $L^{p,\infty}(\mathcal{X})$} is defined to be the set of all
measurable functions $f$ such that
$$\|f\|_{L^{p,\infty}(\mathcal{X})}:=\sup_{\lambda\in(0,\fz)}\lambda
|\{x\in\mathcal{X}:\,\,|f(x)|>\lambda\}|^\frac1p<\fz.$$

\begin{theorem}\label{JohnNirenberg}
  Let $p\in(1,\infty)$, $s\in\zz_+$, $\alpha\in[0,\infty)$ and a closed cube $Q_0\subsetneqq\rn$.
  If $f\in JN_{(p,1,s)_\alpha}(Q_0)$, then
  $f-P_{Q_0}^{(s)}(f)\in L^{p,\infty}(Q_0)$ and there exists a positive constant
  $C_{(n,p,s)}$, depending only on $n$, $p$ and $s$, but independent of $f$, such that
  \begin{align}\label{JohnNirenbergforp1as}
  \lf\|f-P_{Q_0}^{(s)}(f)\r\|_{L^{p,\infty}(Q_0)}\le C_{(n,p,s)} \lf| Q_0 \r|^\alpha
    \|f\|_{JN_{(p,1,s)_\alpha}(Q_0)}.
  \end{align}
\end{theorem}

\begin{remark}\label{jn}
\begin{itemize}
\item[(i)]
  When $p\in(1,\fz)$, $\az=0=s$ and $q=1$, Theorem \ref{JohnNirenberg} is exactly the
  famous John--Nirenberg type inequality \cite[Lemma 3]{JN61}.
  Recall that the John--Nirenberg type inequality \cite[Lemma 3]{JN61} states that,
  for any $\lambda\in(0,\infty)$,
  $$\lambda \lf| \{x\in Q_0:\,\,|f(x)-f_{Q_0}|>\lambda\} \r|^\frac1p
  \le C \|f\|_{JN_{(p,1,0)_0}(Q_0)},$$
  namely, $f-f_{Q_0}\in L^{p,\infty}(Q_0)$ and
  $\|f-f_{Q_0}\|_{L^{p,\infty}(Q_0)}\le C
  \|f\|_{JN_{(p,1,0)_0}(Q_0)}$, where the positive constant $C$
  depends only on $n$ and $p$, but independent of $f$.
  Thus, Theorem \ref{JohnNirenberg} provides a fine generalization of \cite[Lemma 3]{JN61}.

\item[(ii)]
  We can reformulate \eqref{JohnNirenbergforp1as} as follows: for any $\lambda\in(0,\infty)$,
  \begin{align*}
  &\lambda\lf| \lf\{x\in Q_0:\,\,|Q_0|^{-\alpha}\lf|f-P_{Q_0}^{(s)}(f)\r|>\lambda\r\} \r|^\frac1p\\
  &\quad=|Q_0|^{-\alpha}(|Q_0|^\alpha\lambda)\lf| \lf\{x\in Q_0:\,\,
    \lf|f-P_{Q_0}^{(s)}(f)\r|>|Q_0|^{\alpha}\lambda\r\} \r|^\frac1p\\
  &\quad\le|Q_0|^{-\alpha} C_{(n,p,s)} |Q_0|^{\alpha}\|f\|_{JN_{(p,1,s)_\alpha}(Q_0)}\\
  &\quad=C_{(n,p,s)} \|f\|_{JN_{(p,1,s)_\alpha}(Q_0)},
  \end{align*}
  which can be seen as a counterpart of the John--Nirenberg type inequality for
  Campanato spaces as in
  \cite{Li08, LiangYang13}:
  $$\lf| \lf\{x\in Q_0:\,\,|Q_0|^{-\alpha}\lf|f-P_{Q_0}^{(s)}(f)\r|>\lambda\r\} \r|
  \le C_2e^{-\frac{C_3}{\|f\|_{\mathcal{C}_{1,\alpha,s}(Q_0)}}\lambda},
  \quad \forall\,\lambda\in(0,\infty),$$
  where $C_2$ and $C_3$ are positive constants depending only on the dimension $n$.
\end{itemize}
\end{remark}

Based on Proposition \ref{JNpqa=JNp1a}, we obtain the following property of
$HK_{(u,v,s)_\alpha}(\mathcal{X})$.

\begin{proposition}\label{HKuvb=HKu8b}
  Let $1<u<v\le\infty$, $s\in\zz_+$ and $\alpha\in[0,\fz)$.
  Then $HK_{(u,v,s)_\alpha}(\mathcal{X})=HK_{(u,\infty,s)_\alpha}(\mathcal{X})$
  with equivalent norms.
\end{proposition}

\begin{remark}\label{hk}
 When $\mathcal{X}=Q_0$, $\az=0=s$ and $q=1$, by Proposition \ref{case0},
we know that Proposition \ref{HKuvb=HKu8b} coincides with \cite[Proposition 6.4]{DHKY18}.
\end{remark}

In the very recently article \cite{BB18}, A. Brudnyi and Y. Brudnyi introduced
a class of function spaces, $V_\kappa([0,1]^n)$, for any fixed multi-indices $\kappa$.
Let $\mathcal{P}_s([0,1]^n)$ be the set of all polynomials of degree not greater than $s$
on $[0,1]^n$. For any given $p,\,q\in[1,\infty]$, $\lambda\in\rr$, $s\in\zz_+$ and
$\kappa:=\{p,q,\lambda,s\}$, the \emph{space $V_\kappa([0,1]^n)$} is defined by setting
$$V_\kappa([0,1]^n):=\lf\{f\in L^q([0,1]^n):\,\,\|f\|_{V_\kappa([0,1]^n)}<\fz \r\},$$
where
$$\|f\|_{V_\kappa([0,1]^n)}:=\sup\lf\{\sum_i\lf[|Q_i|^{-\lambda}
\inf_{m\in \mathcal{P}_s(Q_i)}\|f-m\|_{L^q(Q_i)} \r]^p \r\}^{\frac1p}$$
with the supremum is taken over all collections of pairwise disjoint cubes
$\{Q_i\}_i$ in $[0,1]^n$ (see \cite[Definition 1.2]{BB18}).
We have the following equivalent relations.

\begin{proposition}\label{b}
Let $p\in (1,\fz)$, $q\in [1,\fz)$, $s\in\zz_+$,
$\az\in [0,\frac{s+1}{n}]$ and $\kappa=\{p,q,\alpha+\frac1q-\frac1p,s\}$.
Then the space $V_\kappa([0,1]^n)$ and $JN_{(p,q,s)_\alpha}([0,1]^n)$
coincide with equivalent norms.
\end{proposition}

\begin{remark}
\begin{itemize}
\item[(i)] The predual of $V_\kappa([0,1]^n)$, denoted by $U_\kappa([0,1]^n)$, was given in
\cite[Theorem 2.6]{BB18} via some kind of atoms for any
$\kappa=\{p,q,\alpha+\frac1q-\frac1p,s\}$ satisfying $p\in[1,\infty]$,
$q\in(1,\infty]$, $s\in\zz_+$ and $\alpha\in[0, \frac{s+1}n]$.
But the assumptions in Theorem \ref{duality} are  $p\in(1,\fz)$, $q\in[1,\fz)$,
$s\in\zz_+$ and $\alpha\in[0,\fz)$, a little bit different.

Moreover, the   space $U_\kappa([0,1]^n)$ in \cite{BB18} is defined via \emph{finite}
linear combination of atoms, while to define the predual space $HK_{(p',q',s)_\alpha}(\mathcal{X})$ of $JN_{(p,q,s)_\alpha}(\mathcal{X})$, we do not need the finite sum requirement,
due to the application of the weak-$\ast$ topology.
It is well known that, for the classical Hardy space $H^1(\rn)$,
the finite atomic norm and the norm $\|\cdot\|_{H^1(\rn)}$
are equivalent on the finite atomic Hardy space $H^1_\fin(\rn)$ (see, for instance, \cite{msv08}).
However, it is still unknown whether or not both the norm $\|\cdot\|_{U_\kappa([0,1]^n)}$
and the norm $\|\cdot\|_{HK_{(p',q',s)_\alpha}([0,1]^n)}$ are equivalent on $U_\kappa([0,1]^n)$
for any given $p\in(1,\fz)$, $q\in(1,\fz)$, $s\in\zz_+$ and $\alpha\in[0, \frac{s+1}n]$.

\item[(ii)] By \cite[Theorem 2.7]{BB18}, we know that, for any given
$p\in(1,\infty]$, $q\in(1,\infty)$, $s\in\zz_+$ and $\alpha\in[0, \frac{s+1}n)$,
the dual of the closure of all smooth functions
in $JN_{(p,q,s)_\alpha}([0,1]^n)$ is the space $HK_{(p',q',s)_\alpha}([0,1]^n)$.
It is easy to see that this is also true when $[0,1]^n$ is replaced by
any closed cube $Q_0$ of $\rn$. However, it is still unknown whether
or not this is true when $[0,1]^n$ is replaced by $\rn$.
\end{itemize}
\end{remark}

The remainder of this article is organized as follows.
We prove Proposition \ref{Campanato&JNpqas}, Corollary \ref{p=8},
Propositions \ref{weak-star} and \ref{case0} in Section \ref{section2}.
Section \ref{section3} is devoted to the proofs of Proposition \ref{JNpq=Lq},
Theorem \ref{duality}, and Corollary \ref{HKqq=Lq}.
One key property we need is the density of
$HK_{(u,v,s)_\alpha}^{\mathrm{fin}}(\mathcal{X})$ in $HK_{(u,v,s)_\alpha}(\mathcal{X})$
(see Lemma \ref{dense} below).
We also need  Lemma \ref{Jensen equ}, which can be seen as a special case of
the Jensen inequality.  In Section \ref{section4}, we first prove
Proposition \ref{b}. Then we give the proofs of Propositions \ref{JNpqa=JNp1a} and
\ref{HKuvb=HKu8b}, namely, the independence of $JN_{(p,q,s)_\alpha}(\mathcal{X})$
and $HK_{(u,v,s)_\alpha}(\mathcal{X})$ over their second sub-indices.
A good-$\lambda$ inequality and a John--Nirenberg type inequality
for $JN_p(Q_0)$ with $f_{Q_0}$ replaced by $P_{Q_0}^{(s)}(f)$ are also obtained.

Finally, we make some conventions on notation. Throughout the article,
we denote by $C$ and $\widetilde{C}$ {positive constants} which
are independent of the main parameters, but they may vary from line to
line. Moreover, we use $C_{(\gamma,\ \beta,\ \ldots)}$ to denote a positive constant depending on the indicated
parameters $\gamma,\ \beta,\ \ldots$. Constants with subscripts, such as $C_{0}$ and $A_1$, do
not change in different occurrences. Moreover, the symbol $f\lesssim g$ represents that
$f\le Cg$ for some positive constant $C$. If $f\lesssim g$ and $g\lesssim f$,
we then write $f\sim g$. We also use the following
convention: If $f\le Cg$ and $g=h$ or $g\le h$, we then write $f\ls g\sim h$
or $f\ls g\ls h$, \emph{rather than} $f\ls g=h$
or $f\ls g\le h$. Let $\mathbb{N}:=\{1,\,2,...\}$ and $\mathbb{Z}_+:=\mathbb{N}\cup\{0\}$.
For any $p\in[0,1]$, let $p'$ be its \emph{conjugate index}, that is, $p'$ satisfies  $1/p+1/p'=1$.
For any subset $E$ of $\rn$, we use the \emph{symbol ${\mathbf 1}_E$} to denote its
characteristic function.
For any cube $Q\subset\rn$, the \emph{symbol $\mathcal{P}_s(Q)$} denotes the set
of all polynomials of degree not greater than $s$ on $Q$.

\section{Proofs of Proposition \ref{Campanato&JNpqas}, Corollary \ref{p=8},
 Propositions \ref{weak-star} and \ref{case0}}\label{section2}

We first show Proposition \ref{Campanato&JNpqas} and Corollary \ref{p=8}.

\begin{proof}[Proof of Proposition \ref{Campanato&JNpqas}]
  Let $\widetilde{Q}$ be a cube in $\mathcal{X}$ and $\{Q_i\}_i$ a collection
  of disjoint cubes in $\mathcal{X}$ which contains $\widetilde{Q}$ as its element.
  Then, for any $f\in L^1_{\loc}(\mathcal{X})$,
  \begin{align*}
    \|f\|_{JN_{(p,q,s)_\alpha}(\mathcal{X})}
    &\geq\lf\{\sum_i|Q_i|\lf[|Q_i|^{-\alpha}\lf\{\fint_{Q_i}\lf|f-P_{Q_i}^{(s)}(f)\r|^q\r\}
    ^{\frac 1q}\r]^p\r\}^{\frac 1p}\\
    &\geq \lf| \widetilde{Q}\r|^{\frac1p}
    \lf| \widetilde{Q}\r|^{-\alpha}\lf[\fint_{\widetilde{Q}}\lf|f-
    P_{\widetilde{Q}}^{(s)}(f)\r|^q\r]^{\frac 1q}.
  \end{align*}
  Hence
  $$\liminf_{p\to\infty}\|f\|_{JN_{(p,q,s)_\alpha}(\mathcal{X})}
  \geq\lf| \widetilde{Q}\r|^{-\alpha}\lf[\fint_{\widetilde{Q}}\lf|f-
    P_{\widetilde{Q}}^{(s)}(f)\r|^q\r]^{\frac 1q}.$$
  By the arbitrariness of $\widetilde{Q}$, we obtain
  $\liminf_{p\to\infty}\|f\|_{JN_{(p,q,s)_\alpha}(\mathcal{X})}
  \geq\|f\|_{\mathcal{C}_{\alpha,q,s}(\mathcal{X})}$.

  Now, let  $f\in\cup_{r\in[1,\infty)}\cap_{p\in[r,\infty)}
  JN_{(p,q,s)_\alpha}(\mathcal{X})$. Then there exist some $r_0\in[1,\infty)$
  such that $f\in JN_{(p,q,s)_\alpha}(\mathcal{X})$ for any $p\in[r_0,\infty)$.
  We now show
  $$\limsup_{p\to\infty}\|f\|_{JN_{(p,q,s)_\alpha}(\mathcal{X})}
  \le\|f\|_{\mathcal{C}_{\alpha,q,s}(\mathcal{X})}.$$
  Without loss of generality, we may assume that
  $\|f\|_{\mathcal{C}_{\alpha,q,s}(\mathcal{X})}=1$.
  Then, for any $p\in[r_0,\infty)$, we have
  \begin{align*}
    \|f\|_{JN_{(p,q,s)_\alpha}(\mathcal{X})}^p
    &=\sup\sum_i|Q_i|\lf\{|Q_i|^{-\alpha}\lf[\fint_{Q_i}\lf|f-P_{Q_i}^{(s)}(f)\r|^q\r]
      ^{\frac 1q}\r\}^p\\
    &\le\sup\sum_i|Q_i|\lf\{|Q_i|^{-\alpha}\lf[\fint_{Q_i}\lf|f-P_{Q_i}^{(s)}(f)\r|^q\r]
      ^{\frac 1q}\r\}^{r_0}
     =\|f\|_{JN_{(r_0,q,s)_\alpha}(\mathcal{X})}^{r_0},
  \end{align*}
  where the supremum is taken over all collections of pairwise disjoint cubes $\{Q_i\}_i$
  of $\mathcal{X}$ and  the inequality holds true because
  $|Q_i|^{-\alpha}[\fint_{Q_i}|f-P_{Q_i}^{(s)}(f)|^q]^{\frac 1q}
  \le \|f\|_{\mathcal{C}_{\alpha,q,s}(\mathcal{X})}=1$.
  Letting $p\to\infty$, we have
  $$\limsup_{p\to\infty}\|f\|_{JN_{(p,q,s)_\alpha}(\mathcal{X})}
  \le1=\|f\|_{\mathcal{C}_{\alpha,q,s}(\mathcal{X})}.$$
  This finishes  the proof of Proposition \ref{Campanato&JNpqas}.
\end{proof}

\begin{proof}[Proof of Corollary \ref{p=8}]
  To show this corollary, by Proposition \ref{Campanato&JNpqas}, it suffices to prove that,
  if $f\in\mathcal{C}_{\alpha,q,s}(Q_0)$ with $q\in[1,\infty)$ and $s\in\zz_+$, then,
  for any $p\in(1,\fz)$,
  $$f\in JN_{(p,q,s)_\alpha}(Q_0)\quad\mathrm{and}\quad\|f\|_{\mathcal{C}_{\alpha,q,s}(Q_0)}
  =\lim_{p\to\fz}\|f\|_{JN_{(p,q,s)_\alpha}(Q_0)}.$$
  Indeed, for any $f\in\mathcal{C}_{\alpha,q,s}(Q_0)$ and $p\in(1,\fz)$,
  \begin{align*}
    \|f\|_{JN_{(p,q,s)_\alpha}(Q_0)}^p
    &=\sup\sum_i|Q_i|\lf\{|Q_i|^{-\alpha}\lf[\fint_{Q_i}\lf|f-P_{Q_i}^{(s)}(f)\r|^q\r]
      ^{\frac 1q}\r\}^p\\
    &\le \sup\sum_i|Q_i|\|f\|_{\mathcal{C}_{\alpha,q,s}(Q_0)}^p
     =|Q_0|\|f\|_{\mathcal{C}_{\alpha,q,s}(Q_0)}^p<\fz,
  \end{align*}
  where the supremum is taken over all collections of pairwise disjoint cubes $\{Q_i\}_i$
  of $Q_0$. Thus, $f\in JN_{(p,q,s)_\alpha}(Q_0)$ for any $p\in(1,\fz)$.
  This, combined with Proposition \ref{Campanato&JNpqas},
  then finishes the proof of Corollary \ref{p=8}.
\end{proof}

Next, we prove Propositions \ref{weak-star} and \ref{case0}.

\begin{proof}[Proof of Proposition \ref{weak-star}]
  Let $u\in(1,\infty)$, $v\in (1,\infty]$, $s\in\zz_+$, $\alpha\in[0,\fz)$
  and $g\in \widetilde{HK}_{(u,v,s)_\alpha}(\mathcal{X})$.
  Then there exist disjoint cubes $\{Q_j\}_j$ in $\mathcal{X}$
  such that $g=\sum_j \lambda_ja_j$ for
  some $(u,v,s)_\alpha$-atoms $\{a_j\}_j$ and $\{\lambda_j\}_j\subset\cc$ satisfying
  $$\lf(\sum_j\lf|\lambda_j\r|^{u}\r)^{\frac1{u}}
  \lesssim\|g\|_{\widetilde{HK}_{(u,v,s)_\alpha}(\mathcal{X})}.$$
  From this and the H\"{o}lder inequality, we deduce that,
  for any $f\in JN_{(u',v',s)_\alpha}(\mathcal{X})$,
  \begin{align}\label{2Holder}
    |\langle f,g\rangle|:=&\lf|\int_{\mathcal{X}}fg\r|\le\sum_j\lf|\int_{Q_j}f \lambda_ja_j\r|
    =\sum_j|Q_j|\lf|\fint_{Q_j}\lf[f-P_{Q_j}^{(s)}(f)\r]\lambda_ja_j\r|\\
    \le&\sum_j|Q_j|\lf[\fint_{Q_j}\lf|f-P_{Q_j}^{(s)}(f)\r|^{v'}\r]^{\frac1{v'}}
        \lf[\fint_{Q_j}|\lambda_ja_j|^{v}\r]^{\frac1{v}}\notag\\
    \le&\lf\{\sum_j|Q_j|^{1-p\alpha}\lf[\fint_{Q_j}\lf|f-P_{Q_j}^{(s)}(f)\r|^{v'}\r]
         ^{\frac {u'}{v'}}\r\}^{\frac 1{u'}}
         \lf\{\sum_j|Q_j|^{1+u\alpha}\lf[\fint_{Q_j}|\lambda_ja_j|^{v}\r]
         ^{\frac{u}{v}}\r\}^\frac1{u}\notag\\
    =&\lf\{\sum_i|Q_i|\lf[|Q_i|^{-\alpha}\lf\{\fint_{Q_i}\lf|f-P_{Q_i}^{(s)}(f)\r|^{v'}\r\}
         ^{\frac 1{v'}}\r]^{u'}\r\}^{\frac 1{u'}}
         \lf\{\sum_j|Q_j|\lf[|Q_j|^\alpha\lf(\fint_{Q_j}|\lambda_ja_j|^{v}\r)
         ^{\frac 1{v}}\r]^{u}\r\}^\frac 1{u}\notag\\
    \le&\lf\{\sum_i|Q_i|\lf[|Q_i|^{-\alpha}\lf\{\fint_{Q_i}\lf|f-P_{Q_i}^{(s)}(f)\r|^{v'}\r\}
         ^{\frac 1{v'}}\r]^{u'}\r\}^{\frac 1{u'}}\lf(\sum_j\lf|\lambda_j\r|^{u}\r)^{\frac1{u}}\notag\\
    \lesssim&\|f\|_{JN_{(u',v',s)_\alpha}(\mathcal{X})}
        \|g\|_{\widetilde{HK}_{(u,v,s)_\alpha}(\mathcal{X})}\notag,
  \end{align}
  which implies $\widetilde{HK}_{(u,v,s)_\alpha}(\mathcal{X})\subset
  (JN_{(u',v',s)_\alpha}(\mathcal{X}))^\ast$.
  This finishes the proof of Proposition \ref{weak-star}.
\end{proof}

\begin{proof}[Proof of Proposition \ref{case0}]
  Let $Q_0$ be a closed cube in $\rn$, $Q_0\subsetneqq\rn$, $u\in(1,\fz)$, $v\in(u,\fz]$
  and $g=\sum_{i\in\nn} g_i$ in $(JN_{(u',v',0)_0}(Q_0))^\ast$, where
  $\{g_i\}_{i\in\nn}\subset\widetilde{HK}_{(u,v,0)_0}(Q_0)$
  and $\sum_{i\in\nn} \|g_i\|_{\widetilde{HK}_{(u,v,0)_0}(Q_0)}<\fz$.
  Then, for any $M,\,m\in\nn$, by \eqref{2Holder} and an obvious fact
  $\|f\|_{JN_{(u',v',0)_0}(Q_0)}\le2\|f\|_{L^{u'}(Q_0)}$, we have
  \begin{align*}
    \lf\|\sum_{i=M}^{M+m} g_i \r\|_{L^{u}(Q_0)}
    &=\sup_{\|f\|_{L^{u'}(Q_0)}\le1}\lf|\lf\langle\sum_{i=M}^{M+m} g_i,f \r\rangle \r|
    \le \sup_{\|f\|_{L^{u'}(Q_0)}\le1} \sum_{i=M}^{M+m}
       \|g_i\|_{\widetilde{HK}_{(u,v,0)_0}(Q_0)}\|f\|_{JN_{(u',v',0)_0}(Q_0)}\\
    &\le2 \sum_{i=M}^{M+m}\|g_i\|_{\widetilde{HK}_{(u,v,0)_0}(Q_0)},
  \end{align*}
  which implies that $\{g_i\}_i$ is a Cauchy sequence in $L^u(Q_0)$ and hence
  there exists some $\widetilde{g}\in L^u(Q_0)$ such that $\widetilde{g}=\sum_i g_i$
  in $L^u(Q_0)$. We claim that $g=\widetilde{g}$ in $(JN_{(u',v',0)_0}(Q_0))^\ast$.
  Indeed, for any $f\in JN_{(u',v',0)_0}(Q_0)$ and any $M\in\nn$, by \eqref{2Holder},
  we obtain
  $$\lf|\int_{Q_0}\sum_{i=M}^\fz g_i f \r|
    \le \sum_{i=M}^\fz \lf\|g_i \r\|_{\widetilde{HK}_{(u,v,0)_0}(Q_0)}
        \lf\|f \r\|_{JN_{(u',v',0)_0}(Q_0)}.$$
  From this and $\sum_{i\in\nn} \|g_i\|_{\widetilde{HK}_{(u,v,0)_0}(Q_0)}<\fz$,
  we deduce that, for any $f\in JN_{(u',v',0)_0}(Q_0)$,
  $\sum_i \int_{Q_0}g_i f= \int_{Q_0} \sum_i g_i f$ and hence
  $$\langle\widetilde{g},f \rangle=\int_{Q_0} \sum_i g_i f=\sum_i \int_{Q_0}g_i f
  =\sum_i \langle g_i,f\rangle=\langle g,f\rangle.$$
  This finishes the proof of Proposition \ref{case0}.
\end{proof}

\section{Proofs of Proposition \ref{JNpq=Lq}, Theorem \ref{duality} and
Corollary \ref{HKqq=Lq}}\label{section3}

We first prove Theorem \ref{duality}. To this end, we begin with the following density property.

\begin{lemma}\label{dense}
  Let $u\in(1,\infty)$, $v\in(1,\infty]$, $s\in\zz_+$ and $\alpha\in[0,\fz)$.
  Then $HK_{(u,v,s)_\alpha}^{\mathrm{fin}}(\mathcal{X})$
  is dense in $HK_{(u,v,s)_\alpha}(\mathcal{X})$.
\end{lemma}
\begin{proof}
  Let $g\in HK_{(u,v,s)_\alpha}(\mathcal{X})$. Assume that
  $g=\sum_{i=1}^\infty g_i$ in $(JN_{(u',v',s)_\alpha}(\mathcal{X}))^\ast$ and
  \begin{align*}
  \sum_{i=1}^\infty \|g_i\|_{\widetilde{HK}_{(u,v,s)_\alpha}(\mathcal{X})}\le2
  \|g\|_{HK_{(u,v,s)_\alpha}(\mathcal{X})}<\infty.
  \end{align*}
  Moreover, without loss of generality, we may assume that, for any $i\in\nn$,
  there exist $\{\lambda_{i,j}\}_{j=1}^\fz\subset\cc$ and  $(u,v,s)_\alpha$-atoms
  $\{a_{i,j}\}_{j=1}^\infty$ such that $g_i=\sum_{j=1}^\infty \lambda_{i,j}a_{i,j}$ pointwise and
  \begin{align*}
  \lf(\sum_{j=1}^\infty\lf|\lambda_{i,j}\r|^u\r)^\frac1u
    \le2\lf\|g_i\r\|_{\widetilde{HK}_{(u,v,s)_\alpha}(\mathcal{X})}<\infty.
  \end{align*}
  Then, for any $N,\,M\in\nn$, we have $\sum_{i=1}^N\sum_{j=1}^M \lambda_{i,j}a_{i,j}
   \in HK_{(u,v,s)_\alpha}^{\mathrm{fin}}(\mathcal{X})$
  and
  \begin{align*}
    \lf\|g-\sum_{i=1}^N\sum_{j=1}^M \lambda_{i,j}a_{i,j}\r\|_
     {HK_{(u,v,s)_\alpha}(\mathcal{X})}
    &=\lf\|g-\sum_{i=1}^N g_i\r\|_{HK_{(u,v,s)_\alpha}(\mathcal{X})}
      +\lf\|\sum_{i=1}^N g_i-\sum_{i=1}^N\sum_{j=1}^M \lambda_{i,j}a_{i,j}\r\|_
       {HK_{(u,v,s)_\alpha}(\mathcal{X})}\\
    &\le\sum_{i=N+1}^\infty\lf\|g_i\r\|_{\widetilde{HK}_{(u,v,s)_\alpha}(\mathcal{X})}
      +\lf\|\sum_{i=1}^N\sum_{j=M+1}^\infty \lambda_{i,j}a_{i,j}\r\|_
       {HK_{(u,v,s)_\alpha}(\mathcal{X})}\\
    &\le\sum_{i=N+1}^\infty\lf\|g_i\r\|_{\widetilde{HK}_{(u,v,s)_\alpha}(\mathcal{X})}
      +\sum_{i=1}^N\lf\|\sum_{j=M+1}^\infty \lambda_{i,j}a_{i,j}\r\|_
       {\widetilde{HK}_{(u,v,s)_\alpha}(\mathcal{X})}\\
    &\le\sum_{i=N+1}^\infty\lf\|g_i\r\|_{\widetilde{HK}_{(u,v,s)_\alpha}(\mathcal{X})}
      +\sum_{i=1}^N\lf(\sum_{j=M+1}^\infty\lf|\lambda_{i,j}\r|^u\r)^\frac1u,
  \end{align*}
  which further implies that $HK_{(u,v,s)_\alpha}^{\mathrm{fin}}(\mathcal{X})$
  is dense in $HK_{(u,v,s)_\alpha}(\mathcal{X})$
  and hence completes the proof of Lemma \ref{dense}.
\end{proof}

\begin{lemma}\label{Jensen equ}
  Let $p\in(1,\infty)$, $\frac1p+\frac1{p'}=1$, $\alpha\in[0,\infty)$ and $N\in\nn$.
  Then, for any positive
  constants $\{a_j\}_{j=1}^N$ and $\{b_j\}_{j=1}^N$, there exist positive
  constants $\{\xi_j\}_{j=1}^N$ such that
  \begin{align}\label{ablambda}
    \sum_{j=1}^N a_j^{1+p'\alpha}\xi_j^{p'}=1
    \quad \mathrm{and}\quad
    \lf(\sum_{j=1}^N a_j^{1-p\alpha} b_j^p\r)^\frac1p=\sum_{j=1}^N a_j \xi_j b_j.
  \end{align}
\end{lemma}
\begin{proof}
  We first show that this lemma holds true for $\alpha=0$.
  In this case, \eqref{ablambda} is equivalent to
  \begin{align}\label{ablambda1}
    \sum_{j=1}^N a_j\xi_j^{p'}=1
    \quad \mathrm{and}\quad
    \sum_{j=1}^N a_j \xi_j^{p'} \lf(\frac{b_j}{\xi_j^{p'-1}}\r)^p
    =\lf(\sum_{j=1}^N a_j \xi_j^{p'} \frac{b_j}{\xi_j^{p'-1}}\r)^p.
  \end{align}
  For any $j\in\{1,\dots,N\}$, letting
  $$\xi_j:=\lf(\sum_{j=1}^N a_j b_j^p\r)^{-\frac{1}{p'}}b_j^{\frac{1}{p'-1}},$$
  then $\{\xi_j\}_{j=1}^N$ satisfies, for any $j\in\{1,\dots,N\}$,
  $$\frac{b_j}{\xi_j^{p'-1}}= \lf(\sum_{j=1}^N a_j b_j^p\r)^{\frac{1}{p}}
    \quad\mathrm{and}\quad
    \sum_{j=1}^N a_j\xi_j^{p'}=1,$$
  which further implies that $\{\xi_j\}_{j=1}^N$ is a solution of \eqref{ablambda1}
  and hence of \eqref{ablambda}.

  When $\alpha\in(0,\infty)$, for any $j\in\{1,\dots,N\}$, let
  $A_j:=a_j^{1-p\alpha}$
  and $B_j:=b_j$. Then, applying the above proved conclusion to
  $\{A_j\}_{j=1}^N$ and $\{B_j\}_{j=1}^N$, we know that there exist positive
  constants $\{\Omega_j\}_{j=1}^N$ such that
  $$\sum_{j=1}^N A_j\Omega_j^{p'}=1
    \quad \mathrm{and}\quad
    \lf(\sum_{j=1}^N A_j B_j^p\r)^\frac1p=\sum_{j=1}^N A_j \Omega_j B_j.$$
  Namely,
  $$\sum_{j=1}^N a_j^{1+p'\alpha}\lf(\Omega_j a_j^{-p\alpha}\r)^{p'}=1
    \quad \mathrm{and}\quad
    \lf(\sum_{j=1}^N a_j^{1-p\alpha} b_j^p\r)^\frac1p
    =\sum_{j=1}^N a_j \Omega_j a_j^{-p\alpha} b_j.$$
  Consequently, \eqref{ablambda} holds true with
  $\xi_j:=\Omega_j a_j^{-p\alpha}$ for any $j\in\{1,\dots,N\}$.
  This finishes the proof of Lemma \ref{Jensen equ}.
\end{proof}

Now, we show Theorem \ref{duality}.

\begin{proof}[Proof of Theorem \ref{duality}]
  For any given $p,\,q,\,s,\,\alpha$ as in this theorem, let
  $f\in JN_{(p,q,s)_\alpha}(\mathcal{X})$. Then $p'\in(1,\fz)$ and  $q'\in(1,\infty]$.
  For any given  $g\in HK_{(p',q',s)_\alpha}(\mathcal{X})$,
  assume that $g=\sum_{i=1}^\infty g_i$ in $(JN_{(p,q,s)_\alpha}(\mathcal{X}))^\ast$ satisfying
  $\sum_{i=1}^\infty \|g_i\|_{\widetilde{HK}_{(u,v,s)_\alpha}(\mathcal{X})}\le2
  \|g\|_{HK_{(u,v,s)_\alpha}(\mathcal{X})}$. Then, by \eqref{2Holder},
  we have
  \begin{align*}
  \lf|\mathcal{L}_f(g)\r|
  &=|\langle f,g\rangle|\le \sum_{i=1}^\infty|\langle f,g_i\rangle|
  \lesssim\sum_{i=1}^\infty\|f\|_{JN_{(p,q,s)_\alpha}(\mathcal{X})}
    \|g_i\|_{\widetilde{HK}_{(u,v,s)_\alpha}(\mathcal{X})}\\
  &\lesssim \|f\|_{JN_{(p,q,s)_\alpha}(\mathcal{X})}\|g\|_{HK_{(u,v,s)_\alpha}(\mathcal{X})}.
  \end{align*}
  This implies that
  \begin{align}\label{norm inequ1}
    \lf\|\mathcal{L}_f\r\|_{(HK_{(p',q',s)_\alpha}(\mathcal{X}))^\ast}
    \lesssim\|f\|_{JN_{(p,q,s)_\alpha}(\mathcal{X})},
  \end{align}
  which completes the proof of (i).

  Now let $\mathcal{L}\in(HK_{(p',q',s)_\alpha}(\mathcal{X}))^\ast$.
  Take a sequence $\{Q^{(k)}\}_{k\in\nn}$ of increasing closed cubes such that
  $Q^{(k)}\uparrow \mathcal{X}$ as $k\to\fz$.
  Then it is easy to see that $\mathcal{L}\in (L^{q'}_s(Q^{(k)}))^\ast$.
  We now claim that there exists a function $f$ on $\mathcal{X}$ such that
  \begin{align}\label{L_k}
  \mathcal{L}(g)=\int_{\mathcal{X}}fg,\quad\forall\,g\in L_s^{q'}(Q^{(k)}),
  \quad \forall\,k\in\nn.
  \end{align}
  Indeed, if letting $\mathcal{P}_s(Q^{(i)})$ for any $i\in\nn$ denote the set of
  all polynomials of degree not greater than $s$ on $Q^{(i)}$, then,
  by the Riesz representation theorem, there exists a function
  $f_1\in L_s^q(Q^{(1)})/\mathcal{P}_s(Q^{(1)})$ such that
  $$\mathcal{L}(g)=\int_{Q^{(1)}}f_1g,\quad\forall\,g\in L_s^{q'}(Q^{(1)}).$$
  Meanwhile, there exists a function
  $f_2\in L_s^q(Q^{(2)})/\mathcal{P}_s(Q^{(2)})$ such that
  $$\mathcal{L}(g)=\int_{Q^{(2)}}f_2g,\quad\forall\,g\in L_s^{q'}(Q^{(2)}).$$
  From these and $Q^{(1)}\subset Q^{(2)}$, we deduce that
  $$\int_{Q^{(1)}}\lf(f_1-f_2\r)g=0,\quad\forall\,g\in L_s^{q'}(Q^{(1)})$$
  and hence
  $$\int_{Q^{(1)}}\lf(f_1-f_2\r)\lf[g-P_{Q^{(1)}}^{(s)}(g)\r]=0,
    \quad\forall\,g\in L^{q'}(Q^{(1)}).$$
  Thus, by \eqref{P_Q^s(f)}, we have
  $$\int_{Q^{(1)}}\lf[f_1-f_2-P_{Q^{(1)}}^{(s)}(f_1-f_2)\r]\lf[g-P_{Q^{(1)}}^{(s)}(g)\r]=0,
    \quad\forall\,g\in L^{q'}(Q^{(1)})$$
  and hence
  $$\int_{Q^{(1)}}\lf[f_1-f_2-P_{Q^{(1)}}^{(s)}(f_1-f_2)\r]g=0,
    \quad\forall\,g\in L^{q'}(Q^{(1)}).$$
  This implies that $f_1-f_2=P_{Q^{(1)}}^{(s)}(f_1-f_2)\in \mathcal{P}_s(Q^{(1)})$
  almost everywhere.
  Letting
  $$f(x):=
    \begin{cases}
      f_1(x),  &\forall\,x\in Q^{(1)}, \\
      f_2(x)+P_{Q^{(1)}}^{(s)}(f_1-f_2)(x),\quad &\forall\,x\in Q^{(2)}\backslash Q^{(1)},
    \end{cases}$$
  we then have
  $$f(x)=f_2(x)+P_{Q^{(1)}}^{(s)}(f_1-f_2)(x) \quad \mathrm{almost\,\,every}\quad x\in Q^{(2)}$$
  and
  $$\mathcal{L}(g)=\int_{Q^{(k)}}fg,\quad\forall\,g\in L_s^{q'}(Q^{(k)}),
  \quad \forall\,k\in\{1,2\}.$$
  Repeating the above procedure, we obtain a function $f$ on $\mathcal{X}$
  such that \eqref{L_k} holds true.

  Next, we show that
  \begin{align}\label{restriction}
  \mathcal{L}(g)=\int_\mathcal{X} fg,\quad\forall\,g\in
   HK_{(p',q',s)_\alpha}^{\mathrm{fin}}(\mathcal{X}).
  \end{align}
  Indeed, for any $g\in HK_{(p',q',s)_\alpha}^{\mathrm{fin}}(\mathcal{X})$,
  by Definition \ref{HKuvb}, we know that
  there exist $M\in\nn$, $\{\lambda_m\}_{m=1}^M\subset\cc$ and
  $(p',q',s)_\alpha$-atoms $\{a_m\}_{m=1}^M$
  such that $g=\sum_{m=1}^M \lambda_m a_m$ pointwise. Then, by \eqref{L_k}, we have
  $$\mathcal{L}(g)=\mathcal{L}\lf(\sum_{m=1}^M \lambda_m a_m\r)
  =\sum_{m=1}^M \lambda_m\mathcal{L}(a_m)
  =\sum_{m=1}^M \lambda_m\int_{\mathcal{X}}fa_m
  =\int_{\mathcal{X}}f\sum_{m=1}^M \lambda_m a_m
  =\int_{\mathcal{X}}fg$$
  and hence \eqref{restriction} holds true.

  We now show that $f\in JN_{(p,q,s)_\alpha}(\mathcal{X})$. Let
  $\{Q_j\}_j$ be a collection of disjoint cubes in $\mathcal{X}$. To prove that
  $f\in JN_{(p,q,s)_\alpha}(\mathcal{X})$, it suffices to
  show that there exists a positive constant $C$ such that, for any $N\in\nn$,
  \begin{align}\label{finite sum}
    \lf\{\sum_{j=1}^N|Q_j|\lf[|Q_j|^{-\alpha}\lf\{\fint_{Q_j}\lf|f-P_{Q_j}^{(s)}(f)\r|^q\r\}
      ^{\frac 1q}\r]^p\r\}^{\frac 1p}
    \le C \|\mathcal{L}\|_{(HK_{(p',q',s)_\alpha}(\mathcal{X}))^\ast}.
  \end{align}
  Indeed, for any
  $j\in\{1,\dots,N\}$, by \eqref{P_Q^s(f)}, we obtain
  \begin{align*}
    \lf[\fint_{Q_j}\lf|f-P_{Q_j}^{(s)}(f)\r|^q\r]^{\frac1q}
    &=\sup\lf\{\fint_{Q_j}\lf[f-P_{Q_j}^{(s)}(f)\r]g:
      \,\,g\in L^{q'}(Q_j)\,\,\mathrm{and}\,\,
      \fint_{Q_j}\lf|g\r|^{q'}\le1\r\}\\
    &=\sup\lf\{\fint_{Q_j}\lf[f-P_{Q_j}^{(s)}(f)\r]
    \lf[g-P_{Q_j}^{(s)}(g)\r]:
      \,\,g\in L^{q'}(Q_j)\,\,\mathrm{and}\,\,\fint_{Q_j}\lf|g\r|^{q'}\le1\r\}\\
    &=\sup\lf\{\fint_{Q_j}f
    \lf[g-P_{Q_j}^{(s)}(g)\r]:\,\,g\in L^{q'}(Q_j)
      \,\,\mathrm{and}\,\,\fint_{Q_j}\lf|g\r|^{q'}\le1\r\}
  \end{align*}
  and hence there exists some $\widetilde{g_j}\in L^{q'}(Q_j)$ with
  $\fint_{Q_j}|\widetilde{g_j}|^{q'}\le1$ such that
  \begin{align*}
    \lf[\fint_{Q_j}\lf|f-P_{Q_j}^{(s)}(f)\r|^q\r]^{\frac1q}
    \le2\fint_{Q_j}f\lf[\widetilde{g_j}-P_{Q_j}^{(s)}(\widetilde{g_j})\r].
  \end{align*}
  For any $j\in\{1,\dots,N\}$, let $g_j:=\widetilde{g_j}-P_{Q_j}^{(s)}(\widetilde{g_j})$.
  Then $g_j\in L_s^{q'}(Q_j)$ and, moreover,  by \eqref{C_{(s)}} and the H\"older inequality,
  we have
  \begin{align*}
    \lf(\fint_{Q_j}\lf|g_j\r|^{q'}\r)^\frac1{q'}
    &\le\lf(\fint_{Q_j}\lf|\widetilde{g_j}\r|^{q'}\r)^\frac1{q'}
      +\lf[\fint_{Q_j}\lf|P_{Q_j}^{(s)}(\widetilde{g_j})\r|^{q'}\r]^\frac1{q'}
     \le\lf(\fint_{Q_j}\lf|\widetilde{g_j}\r|^{q'}\r)^\frac1{q'}
      +C_{(s)}\fint_{Q_j}\lf|\widetilde{g_j}\r|\\
    &\le\lf(\fint_{Q_j}\lf|\widetilde{g_j}\r|^{q'}\r)^\frac1{q'}
      +C_{(s)}\lf(\fint_{Q_j}\lf|\widetilde{g_j}\r|^{q'}\r)^\frac1{q'}
     \le1+C_{(s)},
  \end{align*}
  where $C_{(s)}\in[1,\fz)$ is a constant same as in \eqref{C_{(s)}}.
  Consequently, for any $j\in\{1,\dots,N\}$, there exists $g_j\in L_s^{q'}(Q_j)$ with
  \begin{align}\label{q' morm of g_j}
    \lf(\fint_{Q_j}\lf|g_j\r|^{q'}\r)^\frac1{q'}\le1+C_{(s)}
  \end{align}
    such that
  \begin{align}\label{norm representation}
    \lf[\fint_{Q_j}\lf|f-P_{Q_j}^{(s)}(f)\r|^q\r]^{\frac1q}
    \le2\fint_{Q_j}fg_j.
  \end{align}
  Applying Lemma \ref{Jensen equ} with $a_j:=|Q_j|$ and
  $b_j:=[\fint_{Q_j}|f-P_{Q_j}^{(s)}(f)|^q]^{1/q}$,
  we then know that there exist positive numbers
  $\{\xi_j\}_{j=1}^N$ such that
  \begin{align}\label{lambda_j(1)}
    \sum_{j=1}^N \lf|Q_j\r|^{1+p'\alpha}\xi_j^{p'}=1
  \end{align}
  and
  \begin{align}\label{lambda_j(2)}
    \lf\{\sum_{j=1}^N\lf|Q_j\r|\lf[\lf|Q_j\r|^{-\alpha}\lf\{\fint_{Q_j}
       \lf|f-P_{Q_j}^{(s)}(f)\r|^q\r\}^{\frac 1q}\r]^p\r\}^{\frac 1p}
    =\sum_{j=1}^N \lf|Q_j\r|\xi_j\lf[\fint_{Q_j}\lf|f-P_{Q_j}^{(s)}(f)\r|^q\r]^{\frac1q}.
  \end{align}
  Let
  $$g(x):=
  \begin{cases}
    \displaystyle{\sum_{j=1}^N \xi_j g_j(x)},&\displaystyle{\forall\,x\in\bigcup_{j=1}^N Q_j,}\\
    0,&\displaystyle{\forall\,x\in \mathcal{X}\backslash\bigcup_{j=1}^N Q_j.}\\
  \end{cases}$$
  Then $g\in HK_{(p',q',s)_\alpha}^{\mathrm{fin}}(\mathcal{X})$
  because $g_j\in L_s^{q'}(Q_j)$ for any $j\in\{1,\dots,N\}$.
  By this fact, \eqref{lambda_j(2)}, \eqref{norm representation} and \eqref{restriction},
  we have
  \begin{align}\label{finite sum inequ}
    &\lf\{\sum_{j=1}^N\lf|Q_j\r|\lf[\lf|Q_j\r|^{-\alpha}\lf\{\fint_{Q_j}
      \lf|f-P_{Q_j}^{(s)}(f)\r|^q\r\}^{\frac 1q}\r]^p\r\}^{\frac 1p}\\
    &\quad=\sum_{j=1}^N \lf|Q_j\r|\xi_j
       \lf[\fint_{Q_j}\lf|f-P_{Q_j}^{(s)}(f)\r|^q\r]^{\frac1q}
       \le 2\sum_{j=1}^N \lf|Q_j\r|\xi_j\fint_{Q_j}fg_j\noz\\
    &\quad =2\int_{\bigcup_j Q_j}f\sum_{j=1}^N \xi_j g_j
     =2\int_{\mathcal{X}}fg=2\mathcal{L}(g)\notag.
  \end{align}
  For any $j\in\{1,\dots,N\}$, let
  $\widetilde{g_j}:=\frac{\xi_j |Q_j|^{1/q'-1/p'-\alpha}}
  {\|\xi_j g_j\|_{L^{q'}(Q_j)}}g_j$. Then $\{\widetilde{g_j}\}_{j=1}^N$
  are $(p',q',s)_\alpha$-atoms and
  $$\sum_{j=1}^N \xi_j g_j= \sum_{j=1}^N \frac{\|\xi_j g_j\|_{L^{q'}(Q_j)}}
  {|Q_j|^{1/q'-1/p'-\alpha}}\widetilde{g_j}.$$
  From this,  the fact that $g$ is a $(p',q',s)_\alpha$-polymer,
  \eqref{q' morm of g_j} and \eqref{lambda_j(1)}, it follows that
  \begin{align}\label{norm of g}
    \|g\|_{HK_{(p',q',s)_\alpha}(\mathcal{X})}
    &\le\|g\|_{\widetilde{HK}_{(p',q',s)_\alpha}(\mathcal{X})}
      =\lf\|\sum_{j=1}^N \xi_j g_j\r\|_{\widetilde{HK}_{(p',q',s)_\alpha}(\mathcal{X})}\\
     &\le \lf\{\sum_{j=1}^N \lf[ \frac{\|\xi_j g_j\|_{L^{q'}(Q_j)}}{|Q_j|^{1/q'-1/p'-\alpha}}\r]^{p'} \r\}^\frac1{p'}
     = \lf[\sum_{j=1}^N |Q_j|^{1+p'\alpha}\xi_j^{p'}
     \lf(\fint_{Q_j}\lf|g_j\r|^{q'}\r)^{\frac{p'}{q'}}\r]^\frac1{p'}\notag\\
    &\le\lf[1+C_{(s)}\r]\sum_{j=1}^N |Q_j|^{1+p'\alpha}\xi_j^{p'}
    =1+C_{(s)}\notag.
  \end{align}
  Combining \eqref{finite sum inequ} and \eqref{norm of g}, we obtain
  \begin{align*}
    \lf\{\sum_{j=1}^N|Q_j|\lf[|Q_j|^{-\alpha}\lf\{\fint_{Q_j}\lf|f-P_{Q_j}^{(s)}(f)\r|^q\r\}
      ^{\frac 1q}\r]^p\r\}^{\frac 1p}\le2\mathcal{L}(g)
    \le2\lf[1+C_{(s)}\r]\|\mathcal{L}\|_{(HK_{(p',q',s)_\alpha}(\mathcal{X}))^\ast}.
  \end{align*}
  Thus, \eqref{finite sum} holds true with constant $C:=2[1+C_{(s)}]$
  independent of $N$ and hence
  \begin{align}\label{norm of f}
    \|f\|_{JN_{(p,q,s)_\alpha}(\mathcal{X})}
    \le2\lf[1+C_{(s)}\r]\|\mathcal{L}\|_{(HK_{(p',q',s)_\alpha}(\mathcal{X}))^\ast}.
  \end{align}
On the other hand, by \eqref{restriction} and Lemma \ref{dense},
we know that $\mathcal{L}=\mathcal{L}_f$,
where $\mathcal{L}_f$ is the linear functional induced by $f$ in Theorem \ref{duality}(i).
From this and the above proved Theorem \ref{duality}(i), we deduce that
$$\|\mathcal{L}\|_{(HK_{(p',q',s)_\alpha}(\mathcal{X}))^\ast}
\lesssim\|f\|_{JN_{(p,q,s)_\alpha}(\mathcal{X})}.$$
By this and \eqref{norm of f}, we obtain
$\|\mathcal{L}\|_{(HK_{(p',q',s)_\alpha}(\mathcal{X}))^\ast}
\sim\|f\|_{JN_{(p,q,s)_\alpha}(\mathcal{X})}$, which completes the proof of
Theorem \ref{duality}(ii) and hence of Theorem \ref{duality}.
\end{proof}

Next, we show Proposition \ref{JNpq=Lq} and Corollary \ref{HKqq=Lq}.
\begin{proof}[Proof of Proposition \ref{JNpq=Lq}]
Let $p\in(1,\fz)$, $q\in[p,\fz)$ and $f\in L^1_{\loc}(Q_0)$. Then,
for any disjoint cubes $\{Q_i\}_i$ in $Q_0$ and $m\in \mathcal{P}_s(Q_0)$,
by the Jensen inequality and \eqref{C_{(s)}}, we have
\begin{align*}
\sum_i\frac{|Q_i|}{|Q_0|}\lf[ \fint_{Q_i}\lf|f-P_{Q_i}^{(s)}(f) \r|^q \r]^\frac pq
&\le \lf[\sum_i\frac{|Q_i|}{|Q_0|}\fint_{Q_i}\lf|f-P_{Q_i}^{(s)}(f) \r|^q \r]^\frac pq\\
&\le \lf\{\sum_i\frac{|Q_i|}{|Q_0|}2^{q-1}\lf[\fint_{Q_i}\lf|f+m \r|^q
  +\fint_{Q_i}\lf|P_{Q_i}^{(s)}(f)+m \r|^q \r] \r\}^\frac pq\\
&= \lf\{\sum_i\frac{|Q_i|}{|Q_0|}2^{q-1}\lf[\fint_{Q_i}\lf|f+m \r|^q
  +\fint_{Q_i}\lf|P_{Q_i}^{(s)}(f+m) \r|^q \r] \r\}^\frac pq\\
&\le \lf\{\sum_i\frac{|Q_i|}{|Q_0|}2^{q-1}\lf[\fint_{Q_i}\lf|f+m \r|^q
  +C_{(s)}\lf(\fint_{Q_i}\lf|f+m\r|\r)^q \r] \r\}^\frac pq\\
&\le \lf\{\sum_i\frac{|Q_i|}{|Q_0|}2^{q-1}\lf[\fint_{Q_i}\lf|f+m \r|^q
  +C_{(s)}\fint_{Q_i}\lf|f+m\r|^q \r] \r\}^\frac pq\\
&=\lf\{\sum_i\frac{|Q_i|}{|Q_0|}2^{q-1}\lf[1+C_{(s)}\r]\fint_{Q_i}\lf|f+m \r|^q \r\}^\frac pq\\
&=2^{p-\frac pq}\lf[1+C_{(s)}\r]^\frac pq\|f+m\|_{L^q(Q_0,|Q_0|^{-1}dx)},
\end{align*}
which implies that
$$|Q_0|^{-\frac1p}\|f\|_{JN_{(p,q,s)_0}(Q_0)}
\le 2^{p-\frac pq}\lf[1+C_{(s)}\r]^\frac pq \|f\|_{L^q(Q_0,|Q_0|^{-1}dx)/\mathcal{P}_s(Q_0)}$$
and hence $ L^q(Q_0,|Q_0|^{-1}dx)/\mathcal{P}_s(Q_0)\subset|Q_0|^{-\frac1p} JN_{(p,q,s)_0}(Q_0)$.

On the other hand, we know that
\begin{align*}
\|f\|_{L^q(Q_0,|Q_0|^{-1}dx)/\mathcal{P}_s(Q_0)}
&=\inf_{m\in\mathcal{P}_s(Q_0)}\|f+m\|_{L^q(Q_0,|Q_0|^{-1}dx)}
 \le \lf\|f-P_{Q_0}^{(s)}(f) \r\|_{L^q(Q_0,|Q_0|^{-1}dx)}\\
&=\lf[\fint_{Q_0}\lf|f-P_{Q_0}^{(s)}(f) \r|^q \r]^\frac1q
 =|Q_0|^{-\frac1p}\lf\{|Q_0|\lf[\fint_{Q_0}\lf|f-P_{Q_0}^{(s)}(f) \r|^q \r]^\frac pq \r\}^\frac1p\\
&\le |Q_0|^{-\frac1p}\|f\|_{JN_{(p,q,s)_0}(Q_0)},
\end{align*}
which implies that
$|Q_0|^{-\frac1p} JN_{(p,q,s)_0}(Q_0)\subset L^q(Q_0,|Q_0|^{-1}dx)/\mathcal{P}_s(Q_0)$.
This finishes the proof of Proposition \ref{JNpq=Lq}.
\end{proof}

\begin{proof}[Proof of Corollary \ref{HKqq=Lq}]
Let $p\in(1,\fz)$, $q\in[p,\fz)$ and  $g\in L^{q'}_s(Q_0,|Q_0|^{q'-1}dx)$.
Then $\widetilde{g}:=\frac{|Q_0|^{1/q'-1/p'}}{\|g\|_{L^{q'}(Q_0)}}g$
is a $(p',q',s)_0$-atom.
Thus, we obtain
\begin{align*}
\|g\|_{HK_{(p',q',s)_0}(Q_0)}\le \frac{\|g\|_{L^{q'}(Q_0)}}{|Q_0|^{\frac1{q'}-\frac1{p'}}}
=\frac{|Q_0|^{\frac1{q'}-1}\|g\|_{L^{q'}_s(Q_0,|Q_0|^{q'-1}dx)}}{|Q_0|^{\frac1{q'}-\frac1{p'}}}
=|Q_0|^{-\frac1p}\|g\|_{L^{q'}_s(Q_0,|Q_0|^{q'-1}dx)},
\end{align*}
which implies that $L^{q'}_s(Q_0,|Q_0|^{q'-1}dx)\subset |Q_0|^\frac1p HK_{(p',q',s)_0}(Q_0)$.

On the other hand, for any $g\in HK_{(p',q',s)_0}(Q_0)$, by Theorem \ref{duality}
and Proposition \ref{JNpq=Lq}, we know that
\begin{align*}
\|g\|_{L^{q'}(Q_0,|Q_0|^{q'-1}dx)}&=\sup_{\|h\|_{L^q(Q_0,|Q_0|^{q'-1}dx)}\le1}|\langle g,h\rangle|
\le \sup_{\|h\|_{L^q(Q_0,|Q_0|^{q'-1}dx)/\mathcal{P}_s(Q_0)}\le1}|\langle g,h\rangle|\\
&\le \sup_{\|h\|_{|Q_0|^{-1/p} JN_{(p,q,s)_0}(Q_0)}
      \le2^{p- p/q}[1+C_{(s)}]^{p/q}}|\langle g,h\rangle|\\
&\lesssim \sup_{\|h\|_{|Q_0|^{-1/p} JN_{(p,q,s)_0}(Q_0)}
      \le2^{p- p/q}[1+C_{(s)}]^{p/q}}
      \|g\|_{|Q_0|^{1/p} HK_{(p',q',s)_0}(Q_0)}\|h\|_{|Q_0|^{-1/p} JN_{(p,q,s)_0}(Q_0)}\\
&\lesssim \|g\|_{|Q_0|^{1/p} HK_{(p',q',s)_0}(Q_0)},
\end{align*}
which implies that $|Q_0|^\frac1p HK_{(p',q',s)_0}(Q_0)\subset L^{q'}_s(Q_0,|Q_0|^{q'-1}dx)$.
This finishes the proof of Corollary \ref{HKqq=Lq}
\end{proof}

\section{Proofs of Proposition \ref{JNpqa=JNp1a}, Theorem \ref{JohnNirenberg},
Propositions \ref{HKuvb=HKu8b} and \ref{b}}\label{section4}

In this section, we prove Propositions \ref{JNpqa=JNp1a} and \ref{HKuvb=HKu8b},
Theorem \ref{JohnNirenberg} and Proposition \ref{b}.
We first give the proof of Proposition \ref{b}.

\begin{proof}[Proof of Proposition \ref{b}]
Let $p\in (1,\fz)$, $q\in [1,\fz)$, $s\in\zz_+$,
$\az\in [0,\frac{s+1}{n}]$, $\kappa=\{p,q,\alpha+\frac1q-\frac1p,s\}$
and $f\in L^1_{\loc}([0,1]^n)$. Then,
for any cube $Q\subset [0,1]^n$ and $m\in\mathcal{P}_s(Q)$,
by \eqref{C_{(s)}} and the H\"{o}lder inequality, we have
\begin{align*}
    \lf[\fint_{Q}\lf|f-P_{Q}^{(s)}(f)\r|^q\r]^{\frac 1q}
    &\le\lf[\fint_{Q}\lf|f-m\r|^q\r]^{\frac 1q}
        +\lf[\fint_{Q}\lf|m-P_{Q}^{(s)}(f)\r|^q\r]^{\frac 1q}\\
    &=\lf[\fint_{Q}\lf|f-m\r|^q\r]^{\frac 1q}
        +\lf[\fint_{Q}\lf|P_{Q}^{(s)}(m-f)\r|^q\r]^{\frac 1q}\\
    &\le\lf(\fint_{Q}\lf|f-m\r|^q\r)^{\frac 1q}
        +C_{(s)}\fint_{Q}\lf|m-f\r|
    \le\lf[1+C_{(s)}\r]\lf(\fint_{Q}\lf|f-m\r|^q\r)^{\frac 1q},
\end{align*}
where $C_{(s)}\in[1,\fz)$ is a constant same as in \eqref{C_{(s)}}.
Thus,
$$\lf[\fint_{Q}\lf|f-P_{Q}^{(s)}(f)\r|^q\r]^{\frac 1q}
  \sim\inf_{m\in\mathcal{P}_s(Q)}\lf(\fint_{Q}\lf|f-m\r|^q\r)^{\frac 1q},
  \quad\forall\,Q\subset [0,1]^n$$
and hence
$$\|f\|_{JN_{(p,q,s)_\alpha}([0,1]^n)}\sim
\sup\lf\{\sum_i|Q_i|\lf[|Q_i|^{-\alpha}
  \inf_{m\in\mathcal{P}_s(Q_i)}\lf(\fint_{Q_i}\lf|f-m\r|^q\r)
     ^{\frac 1q}\r]^p\r\}^{\frac 1p}.$$
This implies that $JN_{(p,q,s)_\alpha}([0,1]^n)$ and
$V_\kappa([0,1]^n)$ coincide with equivalent norms, which completes the proof
of Proposition \ref{b}.
\end{proof}

Recall that, for any closed cube $Q$, the \emph{dyadic system} $\mathcal{D}_Q$ on $Q$
is defined by setting
$$\mathcal{D}_Q:=\bigcup_{\ell\in\zz_+}\mathcal{D}_Q^{(\ell)},$$
where, for any $\ell\in\zz_+$,
\begin{align*}
\mathcal{D}_Q^{(\ell)}:=&\lf\{(x_1,\dots,x_n)\in\rn:\,\,
\mathrm{for\,\,any\,\,}i\in\{1,\dots,n\},\,\,x_i\in\lf[a_i+k_i2^{-\ell}l(Q),a_i+(k_i+1)2^{-\ell}l(Q)\r)\r.\\
&\quad\quad\quad\lf.\,\mathrm{when}\,\,
k_i\in\{0,1,\dots,2^\ell-2\}\ \mathrm{or}\ x_i\in\lf[a_i+(1-2^{-\ell})l(Q),a_i+l(Q)\r]\r\}
\end{align*}
and $(a_1,\dots,a_n)$ is a left lower vertex of $Q$ which means that, for any $(x_1,\dots,x_n)\in Q$,
$x_i\ge a_i$ for any $i\in\{1,\dots,n\}$.
In what follows, for any cube $Q\subset\rn$, $\mathcal{M}_{Q}^{(d)}$ denotes
the \emph{dyadic maximal function} related to the dyadic system $\mathcal{D}_Q$ on $Q$, namely,
for any $f\in L^1(Q)$ and $x\in Q$,
$$\mathcal{M}_{Q}^{(d)}(f)(x):=\sup_{Q_{(x)}\ni x}\fint_{Q_{(x)}}|f(y)|\,dy,$$
where the supremum is taken over all dyadic cubes $Q_{(x)}\in \mathcal{D}_Q$ containing $x$.
The following Calder\'{o}n--Zygmund decomposition \cite[p.\,150, Lemma 1]{Stein93}
is needed in the proofs of Theorem \ref{JohnNirenberg} and Proposition \ref{HKuvb=HKu8b}.

\begin{lemma}\label{CZdecomposition}
  Let a closed cube $Q_0\subsetneqq\rn$, $f\in L^1(Q_0)$ and $\lambda\geq\fint_{Q_0}|f|$.
  Then there exist disjoint dyadic cubes $\{Q_k\}_k\subset \mathcal{D}_{Q_0}$ such that
 \begin{enumerate}
  \item[{\rm(i)}]$ \lf\{ x\in Q_0:\,\,\mathcal{M}_{Q_0}^{(d)}(f)(x)
      >\lambda \r\}=\bigcup_k Q_k$;
  \item[{\rm(ii)}]$\lambda<\fint_{Q_k}|f|\le2^n\lambda,\quad \forall\,k$;
  \item[{\rm(iii)}]$\lf| \lf\{ x\in Q_0:\,\,\mathcal{M}_{Q_0}^{(d)}(f)(x)
      >\lambda \r\} \r|\le \frac1\lambda\int_{\{ x\in Q_0:\,\,
      \mathcal{M}_{Q_0}^{(d)}(f)(x)>\lambda \}}|f|$;
  \item[{\rm(iv)}]$f(x)\le \lambda\quad$  almost every $x\in Q_0\backslash\bigcup_k Q_k$.
  \end{enumerate}
\end{lemma}

Now, we prove Theorem \ref{JohnNirenberg} and begin with the
following \emph{good-$\lambda$ inequality}. We employ some ideas used in the proof of
\cite[Lemma 4.5]{ABKY11} with suitable modifications.

\begin{lemma}\label{goodlambda}
  Let $p\in(1,\infty)$, $s\in\zz_+$, $\theta\in(0,2^{-n})$, a closed cube
  $Q_0\subsetneqq\rn$ and $f\in JN_{(p,1,s)_0}(Q_0)$.
  Then, for any $\lambda>\frac1\theta\fint_{Q_0}|f-P_{Q_0}^{(s)}(f)|$,
  \begin{align}\label{goodlambda1}
    &\lf| \lf\{ x\in Q_0:\,\,\mathcal{M}_{Q_0}^{(d)}\lf(f-P_{Q_0}^{(s)}(f)\r)(x)
      >\lambda \r\} \r|\\
    &\quad \le \frac{1}{1-2^n\theta C_{(s)}}\frac{\|f\|_{JN_{(p,1,s)_0}(Q_0)}}{\lambda}
      \lf| \lf\{ x\in Q_0:\,\,\mathcal{M}_{Q_0}^{(d)}\lf(f-P_{Q_0}^{(s)}(f)\r)(x)
      >\theta\lambda \r\} \r|^\frac1{p'},\notag
  \end{align}
  where $C_{(s)}$ is the positive constant same as in \eqref{C_{(s)}}.
\end{lemma}
\begin{proof}
  Without loss of generality, we may assume that $P_{Q_0}^{(s)}(f)=0$, otherwise
  we could replace $f$ by $g:=f-P_{Q_0}^{(s)}(f)$ because $f$ coincides with $g$ in
  $JN_{(p,1,s)_0}(Q_0)$. Then, to show this lemma, it suffices to prove that,
  for any $\lambda>\frac1\theta\fint_{Q_0}|f|$,
  \begin{align}\label{goodlambda2}
    &\lf| \lf\{ x\in Q_0:\,\,\mathcal{M}_{Q_0}^{(d)}(f)(x)
      >\lambda \r\} \r|\\
    &\quad \le \frac{1}{1-2^n\theta C_{(s)}}\frac{\|f\|_{JN_{(p,1,s)_0}(Q_0)}}{\lambda}
      \lf| \lf\{ x\in Q_0:\,\,\mathcal{M}_{Q_0}^{(d)}(f)(x)
      >\theta\lambda \r\} \r|^\frac1{p'}.\notag
  \end{align}
  Applying Lemma \ref{CZdecomposition} to $f$ on $Q_0$ at height $\theta\lambda$,
  we can find disjoint dyadic cubes $\{Q_k\}_k\subset\mathcal{D}_{Q_0}$ such that
  $\{ x\in Q_0:\,\,\mathcal{M}_{Q_0}^{(d)}(f)(x)>\theta\lambda \}
  =\cup_k Q_k$. Since $\theta\in(0,2^{-n})$, it follows that
  $$\lf\{ x\in Q_0:\,\,\mathcal{M}_{Q_0}^{(d)}(f)(x)>\lambda \r\}
  \subset \lf\{ x\in Q_0:\,\,\mathcal{M}_{Q_0}^{(d)}(f)(x)>\theta\lambda \r\}$$
  and hence
  \begin{align}\label{goodlambda3}
    \lf\{ x\in Q_0:\,\,\mathcal{M}_{Q_0}^{(d)}(f)(x)>\lambda \r\}
    =\bigcup_k \lf\{ x\in Q_k:\,\,\mathcal{M}_{Q_0}^{(d)}(f)(x)>\lambda \r\}.
  \end{align}
  We now claim that, for any $k$,
  \begin{align}\label{goodlambda4}
    \lf\{ x\in Q_k:\,\,\mathcal{M}_{Q_0}^{(d)}(f)(x)>\lambda \r\}
    \subset \lf\{ x\in Q_k:\,\,\mathcal{M}_{Q_0}^{(d)}
    \lf( \lf[f-P_{Q_k}^{(s)}(f) \r]{\mathbf 1}_{Q_k} \r)(x)
      >\lf[1-2^n\theta C_{(s)}\r]\lambda \r\}.
  \end{align}
  Indeed, for any $x\in Q_k$ with $\mathcal{M}_{Q_0}^{(d)}(f)(x)>\lambda$,
  by Lemma \ref{CZdecomposition},
  there exists a dyadic cube $Q_{(x)}\ni x$ in $\mathcal{D}_{Q_0}$ such that
  \begin{align}\label{goodlambda5}
    \fint_{Q_{(x)}} |f|>\lambda.
  \end{align}
  Since $Q_k$ is the maximal dyadic cube satisfying
  $\fint_{Q}|f|>\theta\lambda$, then $Q_{(x)}\subset Q_k$. From this and
  \eqref{goodlambda5}, it follows that
  $$\mathcal{M}_{Q_0}^{(d)}\lf(f{\mathbf 1}_{Q_k}\r)(x)
  \geq\fint_{Q_{(x)}} |f|>\lambda.$$
  By this, \eqref{C_{(s)}} and Lemma \ref{CZdecomposition}(ii), we conclude that
  \begin{align*}
    \lambda&<\mathcal{M}_{Q_0}^{(d)}\lf(f{\mathbf 1}_{Q_k}\r)(x)
    =\sup_{Q\ni x}\fint_{Q}\lf|f(y){\mathbf 1}_{Q_k}(y)\r|\,dy\\
    &\le \sup_{Q\ni x} \lf[ \fint_{Q}\lf|f(y)-P_{Q_k}^{(s)}(f)(y) \r|
      {\mathbf 1}_{Q_k}(y)\,dy+\fint_{Q}\lf|P_{Q_k}^{(s)}(f)(y) \r|
      {\mathbf 1}_{Q_k}(y)\,dy\r]\\
    &\le \mathcal{M}_{Q_0}^{(d)} \lf( \lf[f-P_{Q_k}^{(s)}(f) \r]
      {\mathbf 1}_{Q_k} \r)(x) +C_{(s)} \fint_{Q_k}|f|\\
    &\le \mathcal{M}_{Q_0}^{(d)} \lf( \lf[f-P_{Q_k}^{(s)}(f) \r]
      {\mathbf 1}_{Q_k} \r)(x) +2^n \theta C_{(s)} \lambda.
  \end{align*}
  This finishes the proof of \eqref{goodlambda4}.

  Next, we show that, for any $k$,
  \begin{align}\label{goodlambda6}
    \lf| \lf\{ x\in Q_k:\,\,\mathcal{M}_{Q_0}^{(d)}(f)(x)>\lambda \r\} \r|
    \le \frac{1}{[1-2^n\theta C_{(s)}]\lambda}\int_{Q_k}\lf|f-P_{Q_k}^{(s)}(f) \r|.
  \end{align}
  Indeed, if $\fint_{Q_k}|f-P_{Q_k}^{(s)}(f)|>[1-2^n\theta C_{(s)}]\lambda$, then
  $$\lf| \lf\{ x\in Q_k:\,\,\mathcal{M}_{Q_0}^{(d)}(f)(x)>\lambda \r\} \r|
  \le |Q_k|<\frac{1}{[1-2^n\theta C_{(s)}]\lambda}\int_{Q_k}\lf|f-P_{Q_k}^{(s)}(f) \r|.$$
  If $\fint_{Q_k}|f-P_{Q_k}^{(s)}(f)|\le[1-2^n\theta C_{(s)}]\lambda$,
  then, applying Lemma \ref{CZdecomposition} to $f-P_{Q_k}^{(s)}(f)$ on $Q_k$
  at height $[1-2^n\theta C_{(s)}]\lambda$, we obtain
  $$\lf| \lf\{ x\in Q_k:\,\,\mathcal{M}_{Q_k}^{(d)}
     \lf( \lf[f-P_{Q_k}^{(s)}(f){\mathbf 1}_{Q_k} \r] \r)(x)
     >\lf[1-2^n\theta C_{(s)}\r]\lambda \r\} \r|
    \le \frac{1}{[1-2^n\theta C_{(s)}]\lambda}\int_{Q_k}\lf|f-P_{Q_k}^{(s)}(f) \r|.$$
  From this and \eqref{goodlambda4}, we deduce that
  \begin{align*}
    \lf| \lf\{ x\in Q_k:\,\,\mathcal{M}_{Q_0}^{(d)}(f)(x)>\lambda \r\} \r|
    &\le \lf| \lf\{ x\in Q_k:\,\,\mathcal{M}_{Q_0}^{(d)}\lf(
      \lf[f-P_{Q_k}^{(s)}(f){\mathbf 1}_{Q_k} \r] \r)(x)
      >\lf[1-2^n\theta C_{(s)}\r]\lambda \r\} \r|\\
    &= \lf| \lf\{ x\in Q_k:\,\,\mathcal{M}_{Q_k}^{(d)}\lf(
      \lf[f-P_{Q_k}^{(s)}(f){\mathbf 1}_{Q_k} \r] \r)(x)
      >\lf[1-2^n\theta C_{(s)}\r]\lambda \r\} \r|\\
    &\le \frac{1}{[1-2^n\theta C_{(s)}]\lambda}\int_{Q_k}\lf|f-P_{Q_k}^{(s)}(f) \r|.
  \end{align*}
  Thus, \eqref{goodlambda6} holds true.
  From \eqref{goodlambda3}, \eqref{goodlambda6}, the H\"{o}lder inequality and
  the construction of $\{Q_k\}_k$, we deduce that
  \begin{align*}
    &\lf| \lf\{ x\in Q_0:\,\,\mathcal{M}_{Q_0}^{(d)}(f)(x)>\lambda \r\}\r|\\
    &\quad\le \sum_k \lf| \lf\{ x\in Q_k:\,\,\mathcal{M}_{Q_0}^{(d)}(f)(x)>\lambda \r\}\r|\\
    &\quad\le \sum_k \frac{|Q_k|^{\frac1{p'}}}{[1-2^n\theta C_{(s)}]\lambda}
           \lf[|Q_k|^{\frac1p-1} \int_{Q_k}\lf|f-P_{Q_k}^{(s)}(f) \r|\r]\\
    &\quad\le \frac{1}{[1-2^n\theta C_{(s)}]\lambda} \lf( \sum_k |Q_k| \r)^\frac{1}{p'}
           \lf\{\sum_k|Q_k|^{1-p} \lf[\int_{Q_k}\lf|f-P_{Q_k}^{(s)}(f) \r|\r]^p\r\}^\frac1p\\
    &\quad\le \frac{1}{[1-2^n\theta C_{(s)}]\lambda} \lf| \lf\{
       x\in Q_0:\,\,\mathcal{M}_{Q_0}^{(d)}(f)(x)>\theta\lambda
      \r\} \r|^\frac1{p'} \|f\|_{JN_{(p,1,s)_0}(Q_0)},
  \end{align*}
  which shows that \eqref{goodlambda2} holds true.
  This finishes the proof of Lemma \ref{goodlambda}.
\end{proof}

Next, we show Theorem \ref{JohnNirenberg} via using Lemma \ref{goodlambda}.

\begin{proof}[Proof of Theorem \ref{JohnNirenberg}]
  Without loss of generality, we may assume that $P_{Q_0}^{(s)}(f)=0$ as in the proof
  of Lemma \ref{goodlambda}.
  We first prove \eqref{JohnNirenbergforp1as} for $\alpha=0$.
  Let $\theta:=2^{-(n+1)}C_{(s)}^{-1}$, where $C_{(s)}$ is the same as in \eqref{C_{(s)}},
  and $\eta:=\frac{\|f\|_{JN_{(p,1,s)_0}(Q_0)}}{\theta|Q_0|^{1/p}}$. We show that
  $$\|f\|_{L^{p,\infty}(Q_0)}=\sup_{\lambda>0} \lambda \lf| \{ x\in Q_0:\,\,
  |f(x)|>\lambda \} \r|^\frac1p\lesssim\|f\|_{JN_{(p,1,s)_0}(Q_0)}$$
  by considering the following two cases.

  \emph{Case (i)} $\lambda\le\eta$, namely,
  $\lambda\le\frac{\|f\|_{JN_{(p,1,s)_0}(Q_0)}}{\theta|Q_0|^{1/p}}$.
  In this case, $\lambda|Q_0|^\frac1p\le2^{n+1}C_{(s)}\|f\|_{JN_{(p,1,s)_0}(Q_0)}$
  and hence
  $$\sup_{\lambda\in(0,\eta]} \lambda \lf| \{ x\in Q_0:\,\,|f(x)|>\lambda \} \r|^\frac1p
  \le \sup_{\lambda\in(0,\eta]} \lambda |Q_0|^\frac1p
  \le2^{n+1}C_{(s)}\|f\|_{JN_{(p,1,s)_0}(Q_0)}.$$

  \emph{Case (ii)} $\lambda>\eta$. In this case, by the definition of
  $\|f\|_{JN_{(p,1,s)_0}(Q_0)}$ and $\theta<1$,
  we have
  $$\lambda>\eta=\frac{\|f\|_{JN_{(p,1,s)_0}(Q_0)}}{\theta|Q_0|^{\frac1p}}
  \geq\frac{|Q_0|^\frac1p\fint_{Q_0}|f|}{\theta|Q_0|^{\frac1p}}>\fint_{Q_0}|f|.$$
  Next, we show that
  \begin{align}\label{JNinequ0}
  \sup_{\lambda\in(\eta,\infty)}\lambda
  \lf|\lf\{ x\in Q_0:\,\,\mathcal{M}_{Q_0}^{(d)}(f)(x)>\lambda \r\}\r|^\frac1p
  \lesssim\|f\|_{JN_{(p,1,s)_0}(Q_0)}.
  \end{align}
  Let $j_0$ be the smallest nonnegative integer such that $\theta^{-j}\eta<\lambda$.
  By \eqref{goodlambda2}, $\theta\lambda\le\theta^{-j_0}\eta<\lambda$,
  and Lemma \ref{CZdecomposition}(iii), we have
\begin{align*}
    &\lf|\lf\{ x\in Q_0:\,\,\mathcal{M}_{Q_0}^{(d)}(f)(x)>\lambda \r\}\r|\\
    &\quad\le\lf|\lf\{ x\in Q_0:\,\,\mathcal{M}_{Q_0}^{(d)}(f)(x)>\theta^{-j_0}\eta \r\}\r|\\
    &\quad\le\frac{2\|f\|_{JN_{(p,1,s)_0}(Q_0)}}{\theta^{-j_0}\eta}
         \lf|\lf\{ x\in Q_0:\,\,\mathcal{M}_{Q_0}^{(d)}(f)(x)>\theta^{-j_0+1}\eta \r\}\r|
         ^{(p')^{-1}}\\
    &\quad\le\frac{2\|f\|_{JN_{(p,1,s)_0}(Q_0)}}{\theta^{-j_0}\eta}
         \lf[ \frac{2\|f\|_{JN_{(p,1,s)_0}(Q_0)}}{\theta^{-j_0+1}\eta} \r]^{(p')^{-1}}
         \lf|\lf\{ x\in Q_0:\,\,\mathcal{M}_{Q_0}^{(d)}(f)(x)>\theta^{-j_0+2}\eta \r\}\r|
         ^{(p')^{-2}}\\
    &\quad\le\frac{2\|f\|_{JN_{(p,1,s)_0}(Q_0)}}{\theta^{-j_0}\eta}
         \lf[ \frac{2\|f\|_{JN_{(p,1,s)_0}(Q_0)}}{\theta^{-j_0+1}\eta} \r]^{(p')^{-1}}\cdots
         \lf[ \frac{2\|f\|_{JN_{(p,1,s)_0}(Q_0)}}{\theta^{-1}\eta} \r]^{(p')^{-j_0+1}}\\
    &\quad\quad\times\lf|\lf\{ x\in Q_0:\,\,\mathcal{M}_{Q_0}^{(d)}(f)(x)>
         \eta \r\}\r|^{(p')^{-j_0}}\\
    &\quad\le\frac{2\|f\|_{JN_{(p,1,s)_0}(Q_0)}}{\theta\lambda}
         \lf[ \frac{2\|f\|_{JN_{(p,1,s)_0}(Q_0)}}{\theta^{2}\lambda} \r]^{(p')^{-1}}\cdots
         \lf[ \frac{2\|f\|_{JN_{(p,1,s)_0}(Q_0)}}{\theta^{j_0}\lambda} \r]^{(p')^{-j_0+1}}
         \lf( \frac1\eta\int_{Q_0}|f| \r)^{(p')^{-j_0}}\\
    &\quad=\lf(\frac1\theta\r)^{1+2(p')^{-1}+\cdots+j_0(p')^{-j_0}}
         \lf[ \frac{2\|f\|_{JN_{(p,1,s)_0}(Q_0)}}{\lambda}\r]
               ^{1+(p')^{-1}+(p')^{-2}+\cdots+(p')^{-j_0+1}}
         \lf( \frac1\eta\int_{Q_0}|f| \r)^{(p')^{-j_0}}.
  \end{align*}
  Observe that $|Q_0|^\frac1p\fint_{Q_0}|f|\le\|f\|_{JN_{(p,1,s)_0}(Q_0)}$.
  We then obtain
  $$\frac1\eta\int_{Q_0}|f|=\frac{\theta|Q_0|^{\frac1p}}{\|f\|_{JN_{(p,1,s)_0}(Q_0)}}
  \int_{Q_0}|f|\le\theta|Q_0|.$$
  From this,
  $1+2(p')^{-1}+\cdots+j_0(p')^{-j_0}=p^2[1-(p')^{-j_0}]-pj_0(p')^{-j_0}
  \le p^2$ and $$1+(p')^{-1}+(p')^{-2}+\cdots+(p')^{-j_0}=p\lf[1-(p')^{-j_0}\r],$$
  we deduce that
  \begin{align}\label{JNinequ1}
    &\lf|\lf\{ x\in Q_0:\,\,\mathcal{M}_{Q_0}^{(d)}(f)(x)>\lambda \r\}\r|\\
    &\quad\le\lf(\frac1\theta\r)^{p^2}\lf[ \frac{2\|f\|_{JN_{(p,1,s)_0}(Q_0)}}{\lambda}\r]
      ^{p[1-(p')^{-j_0}]}(\theta|Q_0|)^{(p')^{-j_0}}\notag\\
    &\quad=2^{p[1-(p')^{-j_0}]}\lf(\frac1\theta\r)^{p^2-(p')^{-j_0}}
           \lf[ \frac{\|f\|_{JN_{(p,1,s)_0}(Q_0)}}{\lambda}\r]^p
           \lf[\frac{\lambda|Q_0|^\frac1p}{\|f\|_{JN_{(p,1,s)_0}(Q_0)}}\r]
           ^{p(p')^{-j_0}}\notag\\
    &\quad\le2^{p}\lf(\frac1\theta\r)^{p^2}
           \lf[ \frac{\|f\|_{JN_{(p,1,s)_0}(Q_0)}}{\lambda}\r]^p
           \lf[\frac{\lambda|Q_0|^\frac1p}{\|f\|_{JN_{(p,1,s)_0}(Q_0)}}\r]
           ^{p(p')^{-j_0}}.\notag
  \end{align}
  By the definitions of $\eta$ and $j_0$, we have
  \begin{align}\label{JNinequ2}
    \frac{\lambda|Q_0|^\frac1p}{\|f\|_{JN_{(p,1,s)_0}(Q_0)}}
    =\frac{\lambda}{\theta\eta}\le\frac{\theta^{-j_0-1}\eta}{\theta\eta}
    =\theta^{-j_0-2}.
  \end{align}
  We now claim that, for any $j\in\nn$,
  \begin{align}\label{JNinequ3}
    (j+2)(p')^{-j}\le \max \{p,2\}.
  \end{align}
  Indeed, for any $j\in\zz_+$, let $F(j):=(j+2)(p')^{-j}$.
  Then $F$ attains its maximal value at some $j_1\in\zz_+$.
  If $j_1=0$, then $(j_1+2)(p')^{-j_1}=2\le \max \{p,2\}$. If $j_1\in\nn$,
  then
  $$\frac{F(j_1-1)}{F(j_1)}=\frac{(j_1+1)(p')^{-j_1+1}}{(j_1+2)(p')^{-j_1}}
  \le1,$$
  which implies that $j_1+2\le \frac1{p'-1}+1=p$ and hence
  $(j_1+2)(p')^{-j_1}\le p\le \max \{p,2\}$. This proves \eqref{JNinequ3}.
  Therefore, by \eqref{JNinequ1}, \eqref{JNinequ2}, \eqref{JNinequ3} and
  $\theta=2^{-(n+1)}C_{(s)}^{-1}$, we conclude that
  \begin{align*}
    &\lf|\lf\{ x\in Q_0:\,\,\mathcal{M}_{Q_0}^{(d)}(f)(x)>\lambda \r\}\r|\\
    &\quad\le2^{p}\lf(\frac1\theta\r)^{p^2-p(j_0+2)(p')^{-j_0}}
        \lf[ \frac{\|f\|_{JN_{(p,1,s)_0}(Q_0)}}{\lambda}\r]^p\\
    &\quad\le2^{p}\lf(\frac1\theta\r)^{p^2-p\max \{p,2\}}
        \lf[ \frac{\|f\|_{JN_{(p,1,s)_0}(Q_0)}}{\lambda}\r]^p\\
    &\quad=2^{p+(n+1)(p^2-p\max \{p,2\})}C_{(s)}^{p^2-p\max \{p,2\}}
        \lf[ \frac{\|f\|_{JN_{(p,1,s)_0}(Q_0)}}{\lambda}\r]^p.
  \end{align*}
  This implies that \eqref{JNinequ0} holds true.
  Moreover, by Lemma \ref{CZdecomposition}(iv), we have
  $$\lf\{x\in Q_0:\,\,|f(x)|>\lambda\r\}\subset
  \lf\{ x\in Q_0:\,\,\mathcal{M}_{Q_0}^{(d)}(f)(x)>\lambda \r\}.$$
  From this and \eqref{JNinequ0}, it follows that
  $$\sup_{\lambda\in(\eta,\infty)}\lambda
  \lf|\lf\{ x\in Q_0:\,\,|f(x)|>\lambda \r\}\r|^\frac1p
  \lesssim\|f\|_{JN_{(p,1,s)_0}(Q_0)}.$$
  Therefore, \eqref{JohnNirenbergforp1as} for $\alpha=0$ holds true
  by combining Case (i) and Case (ii)
  and letting $$C_{(n,p,s)}:=\max\lf\{2^{n+1}C_{(s)},
  2^{p+(n+1)(p^2-p\max \{p,2\})}C_{(s)}^{p^2-p\max \{p,2\}}\r\}.$$

  Finally, for any $\alpha\in[0,\infty)$, by \eqref{JohnNirenbergforp1as}
  for $\alpha=0$, we find that
  \begin{align*}
  \lf\|f-P_{Q_0}^{(s)}(f)\r\|_{L^{p,\infty}(Q_0)}
      &\le C_{(n,p,s)} \|f\|_{JN_{(p,1,s)_0}(Q_0)}
      =C_{(n,p,s)} \sup\lf\{\sum_i|Q_i|\lf[
      \fint_{Q_i}\lf|f-P_{Q_i}^{(s)}(f)\r|\r]^p\r\}^{\frac 1p}\\
  &\le C_{(n,p,s)} |Q_0|^\alpha\sup\lf\{\sum_i|Q_i|\lf[|Q_i|^{-\alpha}
      \lf\{\fint_{Q_i}\lf|f-P_{Q_i}^{(s)}(f)\r|\r\}\r]^p\r\}^{\frac 1p}\\
  &\le C_{(n,p,s)} |Q_0|^{\alpha}\|f\|_{JN_{(p,1,s)_\alpha}(Q_0)}.
  \end{align*}
  This finishes the proof of Theorem \ref{JohnNirenberg}.
\end{proof}

Finally, we use Theorem \ref{JohnNirenberg} to prove Proposition \ref{JNpqa=JNp1a}.

\begin{proof}[Proof of Proposition \ref{JNpqa=JNp1a}]
  For any $f\in JN_{(p,q,s)_\alpha}(\mathcal{X})$ with $1\le q<p<\infty$ and
  $\alpha\in[0,\infty)$, by the H\"{o}lder inequality, we obtain
  \begin{align*}
    \|f\|_{JN_{(p,1,s)_\alpha}(\mathcal{X})}
    &=\sup\lf\{\sum_i|Q_i|\lf[|Q_i|^{-\alpha}\lf\{\fint_{Q_i}\lf|f-P_{Q_i}^{(s)}(f)\r|\r\}
      \r]^p\r\}^{\frac 1p}\\
    &\le\sup\lf\{\sum_i|Q_i|\lf[|Q_i|^{-\alpha}\lf\{\fint_{Q_i}\lf|f-P_{Q_i}^{(s)}(f)\r|^q\r\}
      ^{\frac 1q}\r]^p\r\}^{\frac 1p}
    =\|f\|_{JN_{(p,q,s)_\alpha}(\mathcal{X})}.
  \end{align*}
  This shows that $JN_{(p,q,s)_\alpha}(\mathcal{X})
  \subset JN_{(p,1,s)_\alpha}(\mathcal{X})$.

  On the other hand, for any $1\le q<p<\infty$ and cube $Q\subset\rn$,
  by the embedding $L^{p,\infty}(Q)\subset L^q(Q)$ (which is easy to prove; see, for instance,
  \cite[p.\,14, Exercises 1.1.11]{GTM249}) and Theorem \ref{JohnNirenberg},
  we have
  \begin{align}\label{qto1}
  \lf[\fint_{Q}\lf|f-P_{Q}^{(s)}(f)\r|^q\r]^\frac 1q
  \lesssim |Q|^{-\frac1 p}\lf\|f-P_{Q}^{(s)}(f)\r\|_{L^{p,\infty}(Q)}
  \lesssim |Q|^{-\frac1 p}\|f\|_{JN_{(p,1,s)_0}(Q)}.
  \end{align}
  Now, for any given disjoint cubes $\{Q_i\}_i$ in $\mathcal{X}$, there exist disjoint cubes
  $\{Q_{i,j}\}_j\subset Q_i$ such that
  \begin{align}\label{attainJNp1aNorm}
    \|f\|_{JN_{(p,1,s)_0}(Q_i)}
    \lesssim\lf\{\sum_j\lf|Q_{i,j}\r|\lf[\fint_{Q_{i,j}}
       \lf|f-P_{Q_{i,j}}^{(s)}(f)\r|\r]^p\r\}^\frac1p,
    \quad\forall\,i.
  \end{align}
  By \eqref{qto1}, \eqref{attainJNp1aNorm} and $\alpha\in[0,\infty)$, we obtain
  \begin{align*}
    &\sum_i|Q_i|\lf\{|Q_i|^{-\alpha}\lf[\fint_{Q_i}\lf|f-P_{Q_i}^{(s)}(f)\r|^q\r]
      ^{\frac 1q}\r\}^p\\
    &\quad\lesssim\sum_i|Q_i|\lf[|Q_i|^{-\alpha}|Q_i|^{-\frac1 p}\|f\|_{JN_{(p,1,s)_0}(Q_i)}
      \r]^p\\
    &\quad\sim\sum_i|Q_i|^{-\alpha p}\|f\|_{JN_{(p,1,s)_0}(Q_i)}^p
     \lesssim\sum_i|Q_i|^{-\alpha p}\sum_j\lf|Q_{i,j}\r|\lf[\fint_{Q_{i,j}}
       \lf|f-P_{Q_{i,j}}^{(s)}(f)\r|\r]^p\\
    &\quad\lesssim\sum_i\sum_j\lf|Q_{i,j}\r|\lf[\lf|Q_{i,j}\r|^{-\alpha}\fint_{Q_{i,j}}
       \lf|f-P_{Q_{i,j}}^{(s)}(f)\r|\r]^p
    \lesssim\|f\|_{JN_{(p,1,s)_\alpha}(\mathcal{X})}^p.
  \end{align*}
  This implies that $\|f\|_{JN_{(p,q,s)_\alpha}(\mathcal{X})}
  \lesssim\|f\|_{JN_{(p,1,s)_\alpha}(\mathcal{X})}$ and hence
  $JN_{(p,1,s)_\alpha}(\mathcal{X})\subset JN_{(p,q,s)_\alpha}(\mathcal{X})$.

  To sum up, $JN_{(p,q,s)_\alpha}(\mathcal{X})=JN_{(p,1,s)_\alpha}(\mathcal{X})$ with
  equivalent norms. This finishes the proof of Proposition \ref{JNpqa=JNp1a}.
\end{proof}

Now, we prove Proposition \ref{HKuvb=HKu8b}. To this end, we begin with the following lemma.

\begin{lemma}\label{Cavalieri}
  Let $(\mathcal{Y},\mu)$ be a measure space, $v\in(1,\infty)$ and
  positive numbers $\{\delta_k\}_{k\in\nn}$ satisfy
  \begin{enumerate}
  \item[{\rm(i)}]$\delta_k\uparrow\infty$\, as\,  $k\to\infty$;
  \item[{\rm(ii)}]$\forall\, m\in\nn$,
  $\sum_{k=1}^m\lf(\frac{\delta_k}{\delta_m}\r)^{v}\le C$
  for some positive constant $C$ independent of $m$.
  \end{enumerate}
  Then, for any $f\in L^{v}(\mathcal{Y})$,
  $$\sum_{k=1}^\infty\delta_k^v\mu(\{x\in\mathcal{Y}:\,\,|f(x)|>\delta_k\})
  \le C\|f\|_{L^v(\mathcal{Y})}^v,$$
  where the positive constant $C$ is same as in (ii).
\end{lemma}
\begin{proof}
  For any $f\in L^{v}(\mathcal{Y})$, by $v\in(1,\infty)$,
  (i) and (ii) of Lemma \ref{Cavalieri}, we have
  \begin{align*}
    &\sum_{k=1}^\infty\delta_k^v\mu(\{x\in\mathcal{Y}:\,\,|f(x)|>\delta_k\})\\
    &\quad=\sum_{k=1}^\infty\delta_k^{v}
    \int_{\{x\in\mathcal{Y}:\,\,|f(x)|>\delta_k\}}1\,d\mu
    =\sum_{k=1}^\infty\delta_k^{v}\sum_{m=k}^\infty
    \int_{\{x\in\mathcal{Y}:\,\,\delta_m<|f(x)|\le\delta_{m+1}\}}1\,d\mu\\
    &\quad=\sum_{m=1}^\infty\sum_{k=1}^m\delta_k^{v}\delta_m^{-v}
    \int_{\{x\in\mathcal{Y}:\,\,\delta_m<|f(x)|\le\delta_{m+1}\}}
    \delta_m^{v}\,d\mu
    \le\sum_{m=1}^\infty\sum_{k=1}^m\lf(\frac{\delta_k}{\delta_m}\r)^{v}
    \int_{\{x\in\mathcal{Y}:\,\,\delta_m<|f(x)|\le\delta_{m+1}\}}
    |f|^{v}\,d\mu\\
    &\quad\lesssim\sum_{m=1}^\infty\int_{\{x\in\mathcal{Y}:\,\,
    \delta_m<|f(x)|\le\delta_{m+1}\}}|f|^{v}\,d\mu
    \lesssim\int_{\mathcal{Y}}|f|^{v}\,d\mu
    \sim\|f\|_{L^{v}(\mathcal{Y})}^v.
  \end{align*}
  This finishes the proof of Lemma \ref{Cavalieri}.
\end{proof}

The next decomposition lemma is a generalization of \cite[Lemma 6.5]{DHKY18}
which corresponds to the case $s=0$.

\begin{lemma}\label{atomdecomposition}
  Let $C_1\in(2^n,\infty)$, $s\in\zz_+$, a closed cube $Q_0\subsetneqq\rn$,
  $f\in L_s^1(Q_0)$ and  $\lambda\geq\fint_{Q_0}|f|$. Then
  $$f-P_{Q_0}^{(s)}(f)=\sum_{k=0}^\infty\sum_j a_{k,j}$$
  almost everywhere, where $a_{k,j}\in L_s^\infty(Q_{k,j})$
  satisfies
  $\|a_{k,j}\|_{L^\infty(Q_0)}\le 2^{n+1} C_{(s)} C_1^{k+1}\lambda$
  for any $j$ with $C_{(s)}$ same as in \eqref{C_{(s)}},
  $\{Q_{k,j}\}_j$ is a collection of disjoint cubes in $Q_0$, satisfying
  $$\bigcup_j Q_{0,j}=Q_0\quad and\quad
  \bigcup_j Q_{k,j}=\lf\{x\in Q_0:\,\,\mathcal{M}_{Q_0}^{(d)}f(x)
    >C_1^k\lambda\r\},\,\,\forall\,k\in\nn.$$
\end{lemma}
\begin{proof}
  Let $Q_{0,0}:=Q_0$ and $Q_{0,j}:=\emptyset$ for any $j>0$.
  For any $k\in\nn$, applying Lemma \ref{CZdecomposition} to $f$ on $Q_0$
  at height $C_1^k\lambda$,
  we conclude that there exist cubes $\{Q_{k,j}\}_j$ such that
  $$\lf\{x\in Q_0:\,\,\mathcal{M}_{Q_0}^{(d)}f(x)>C_1^k\lambda\r\}
    =\bigcup_j Q_{k,j},$$
  where $\{Q_{k,j}\}_j$ are the maximal dyadic cubes satisfying
  $\fint_{Q_{k,j}}|f|>C_1^k\lambda$. By Lemma \ref{CZdecomposition}(ii)
  and $C_1\in(2^n,\infty)$, we know that $\fint_{Q_{k,j}}|f|\le2^n C_1^k \lambda
  <C_1^{k+1}\lambda$ and hence $Q_{k,j}\notin\{Q_{k+1,i}\}_i$.

  Now, we show that
  \begin{align}\label{fac}
  f(x)-P_{Q_0}^{(s)}(f)(x)=\sum_{k=0}^\infty\sum_j a_{k,j}(x)
    \quad \mathrm{almost\,\,every}\,\,x\in Q_0,
  \end{align}
  where
  $$a_{k,j}:=\lf[f-P_{Q_{k,j}}^{(s)}(f)\r]
    {\mathbf 1}_{Q_{k,j}\backslash\bigcup_i Q_{k+1,i}}
    +\sum_{i:\,\,Q_{k+1,i}\subset Q_{k,j}}
     \lf[P_{Q_{k+1,i}}^{(s)}(f)-P_{Q_{k,j}}^{(s)}(f)\r]
     {\mathbf 1}_{Q_{k+1,i}}.$$
  Indeed, from Lemma \ref{CZdecomposition}(iii), it follows that
  $\mathcal{M}_{Q_0}^{(d)}f(x)<\infty$ for almost every $x\in Q_0$.
  Thus, we have the following equalities, in almost everywhere sense,
  \begin{align*}
    f&=f{\mathbf 1}_{Q_0}=f{\mathbf 1}_{Q_0\backslash
      \{x\in Q_0:\,\,\mathcal{M}_{Q_0}^{(d)}f(x)>C_1\lambda\}}
      +f{\mathbf 1}_{\bigcup_{k=1}^\infty
      \{x\in Q_0:\,\,\mathcal{M}_{Q_0}^{(d)}f(x)>C_1^k\lambda\}
      \backslash\{x\in Q_0:\,\,\mathcal{M}_{Q_0}^{(d)}f(x)>C_1^{k+1}\lambda\}}\\
     &=\sum_{k=0}^\infty f{\mathbf 1}_{\bigcup_j Q_{k,j}\backslash
      \bigcup_i Q_{k+1,i}}
      =\sum_{k=0}^\infty \sum_j f{\mathbf 1}_{Q_{k,j}\backslash
      \bigcup_i Q_{k+1,i}}\\
     &=\sum_{k=0}^\infty \sum_j \lf[f-P_{Q_{k,j}}^{(s)}(f)\r]
       {\mathbf 1}_{Q_{k,j}\backslash\bigcup_i Q_{k+1,i}}
       +\sum_{k=0}^\infty \sum_j P_{Q_{k,j}}^{(s)}(f)
       {\mathbf 1}_{Q_{k,j}\backslash\bigcup_i Q_{k+1,i}}\\
     &=:\sum_{k=0}^\infty \sum_j \lf[f-P_{Q_{k,j}}^{(s)}(f)\r]
       {\mathbf 1}_{Q_{k,j}\backslash\bigcup_i Q_{k+1,i}}
       +\mathrm{I}.
  \end{align*}
  To prove \eqref{fac}, it suffices to show that
  \begin{align}\label{I}
    \mathrm{I}
    =\sum_{k=0}^\infty \sum_j \sum_{i:\,\,Q_{k+1,i}\subset Q_{k,j}}
     \lf[P_{Q_{k+1,i}}^{(s)}(f)-P_{Q_{k,j}}^{(s)}(f)\r]
     {\mathbf 1}_{Q_{k+1,i}}+P_{Q_0}^{(s)}(f)
  \end{align}
  almost everywhere in $Q_0$.

  Indeed, for any
  $x\in Q_0\backslash\cap_{k=0}^\infty
   \{x\in Q_0:\,\,\mathcal{M}_{Q_0}^{(d)}f(x)>C_1^k\lambda\}$,
  there exists a maximal
  $k_x:=\min\{k\in\nn:\,\,\mathcal{M}_{Q_0}^{(d)}f(x)\le C_1^k\lambda\}-1$
  such that $x\in Q_{k_x,j_{k_x}}$
  and $x\notin Q_{k_x+1,i}$ for any $i$. Moreover,
  $x\notin Q_{k,j}$ for any $k\in\{k_x+1,\dots\}$ and hence
  $${\mathbf 1}_{Q_{k,j}\backslash\cup_i Q_{k+1,i}}(x)=0,
    \quad\forall\,k\in\{k_x+1,\dots\}.$$
  On the other hand, when $k_x>0$, since
  $\{\{x\in Q_0:\,\,\mathcal{M}_{Q_0}^{(d)}f(x)>C_1^k\lambda\}\}_{k\in\nn}$ decreases as
  $k$ increases, it then follows that, for any $k\in\{0,\dots,k_x-1\}$, we have
  $\cup_\ell Q_{k_x,\ell}\subset\cup_i Q_{k+1,i}$
  which implies that $x\notin Q_{k,j}\backslash\cup_i Q_{k+1,i}$ and hence
  $${\mathbf 1}_{Q_{k,j}\backslash\cup_i Q_{k+1,i}}(x)=0,
    \quad\forall\,k\in\{0,\dots,k_x-1\}.$$
  Consequently,
  \begin{align*}
    \mathrm{I}(x)&=\sum_{k=0}^\infty \sum_j P_{Q_{k,j}}^{(s)}(f)(x)
    {\mathbf 1}_{Q_{k,j}\backslash\bigcup_i Q_{k+1,i}}(x)\\
    &=\sum_j P_{Q_{k_x,j}}^{(s)}(f)(x)
    {\mathbf 1}_{Q_{k_x,j}\backslash\bigcup_i Q_{k_x+1,i}}(x)
    =P_{Q_{k_x,j_{k_x}}}^{(s)}(f)(x).
  \end{align*}
Moreover, for any $k\in\{0,\dots, k_x\}$, by
  $x\in Q_{k_x,j_{k_x}}\subset\cup_i Q_{k_x,i}\subset\cup_j Q_{k,j}$
  and the disjointness of $\{Q_{k,j}\}_j$, we know that
  there exists a unique $j_k$ such that $x\in Q_{k,j_k}$.
  Therefore,
  \begin{align*}
    &\sum_{k=0}^\infty \sum_j \sum_{\{i:\,\,Q_{k+1,i}\subset Q_{k,j}\}}
     \lf[P_{Q_{k+1,i}}^{(s)}(f)(x)-P_{Q_{k,j}}^{(s)}(f)(x)\r]
     {\mathbf 1}_{Q_{k+1,i}}(x)\\
    &\quad=\sum_{k=0}^{k_x-1} \sum_{\{i:\,\,Q_{k+1,i}\subset Q_{k,j_x}\}}
     \lf[P_{Q_{k+1,i}}^{(s)}(f)(x)-P_{Q_{k,j_x}}^{(s)}(f)(x)\r]
     {\mathbf 1}_{Q_{k+1,i}}(x)\\
    &\quad=\sum_{k=0}^{k_x-1}
     \lf[P_{Q_{k+1,j_{k+1}}}^{(s)}(f)(x)-P_{Q_{k,j_k}}^{(s)}(f)(x)\r]
     {\mathbf 1}_{Q_{k+1,j_{k+1}}}(x)\\
    &\quad=P_{Q_{1,j_1}}^{(s)}(f)(x)-P_{Q_0}^{(s)}(f)(x)
           +P_{Q_{2,j_2}}^{(s)}(f)(x)-P_{Q_{1,j_1}}^{(s)}(f)(x)+\dots\\
    &\quad\quad+P_{Q_{k_x,j_{k_x}}}^{(s)}(f)(x)
                  -P_{Q_{k_x-1,j_{k_x-1}}}^{(s)}(f)(x)\\
    &\quad=P_{Q_{k_x,j_{k_x}}}^{(s)}(f)(x)-P_{Q_0}^{(s)}(f)(x)
     =\mathrm{I}(x)-P_{Q_0}^{(s)}(f)(x).
  \end{align*}
  Thus, \eqref{I} holds true almost everywhere and hence
  $f-P_{Q_0}^{(s)}(f)=\sum_{k=0}^\infty\sum_j a_{k,j}$ holds true almost everywhere.

  Next, we show that $\|a_{k,j}\|_{L^\infty(Q_0)} \le 2^{n+1} C_{(s)} C_1^{k+1}\lambda$
  for any $k\in\zz_+$ and $j$.
  Indeed, when $k=0$,
  $$\lf|P_{Q_{k,j}}^{(s)}(f)\r|\le C_{(s)}\fint_{Q_{k,j}}|f|\le C_{(s)}\fint_{Q_0}|f|
  \le C_{(s)}\lambda = C_{(s)}C_1^k\lambda.$$
  When $k\in\nn$, by (ii) and (iv) of Lemma \ref{CZdecomposition} and \eqref{C_{(s)}},
  for any $j$, we know that
  $$\lf|P_{Q_{k,j}}^{(s)}(f)\r|\le C_{(s)}\fint_{Q_{k,j}}|f|\le C_{(s)}2^nC_1^k\lambda$$
  and $|f(x)|\le C_1^{k+1}\lambda$ for almost every
  $x\in Q_{k,j}\backslash\cup_i Q_{k+1,i}$. These imply that, for any $k\in\zz_+$ and $j$,
  $$\lf\|a_{k,j}\r\|_{L^\infty(Q_0)} \le 2^{n+1} C_{(s)} C_1^{k+1}\lambda.$$
  Finally, it remains to  show that $\int_{Q_{k,j}}a_{k,j}(x)x^\gamma\,dx=0$
  for any $|\gamma|\le s$.
  Since both $Q_{k,j}$ and $\{Q_{k+1,i}\}_i$ are dyadic cubes, it follows that
  $$\lf(Q_{k,j}\backslash\bigcup_i Q_{k+1,i}\r)
   \bigcup\lf(\bigcup_{i:\,\, Q_{k+1,i}\subset Q_{k,j}} Q_{k+1,i}\r)
   =Q_{k,j}.$$
  By this and \eqref{P_Q^s(f)}, we conclude that
  \begin{align*}
    \int_{Q_{k,j}}a_{k,j}(x)x^\gamma\,dx
    &=\int_{Q_{k,j}\backslash\bigcup_i Q_{k+1,i}}
      \lf[f(x)-P_{Q_{k,j}}^{(s)}(f)(x)\r]x^\gamma\,dx\\
    &\quad+\sum_{\{i:\,\, Q_{k+1,i}\subset Q_{k,j}\}}\int_{Q_{k+1,i}}
       \lf[P_{Q_{k+1,i}}^{(s)}(f)(x)-P_{Q_{k,j}}^{(s)}(f)(x)\r]x^\gamma\,dx\\
    &=\int_{Q_{k,j}\backslash\bigcup_i Q_{k+1,i}}f(x)x^\gamma\,dx
      +\sum_{\{i:\,\, Q_{k+1,i}\subset Q_{k,j}\}}\int_{Q_{k+1,i}}
      P_{Q_{k+1,i}}^{(s)}(f)(x)x^\gamma\,dx\\
    &\quad-\int_{Q_{k,j}\backslash\bigcup_i Q_{k+1,i}}P_{Q_{k,j}}^{(s)}(f)(x)
      x^\gamma\,dx-\sum_{\{i:\,\, Q_{k+1,i}\subset Q_{k,j}\}}\int_{Q_{k+1,i}}
      P_{Q_{k,j}}^{(s)}(f)(x)x^\gamma\,dx\\
    &=\int_{Q_{k,j}\backslash\bigcup_i Q_{k+1,i}}f(x)x^\gamma\,dx
      +\sum_{\{i:\,\, Q_{k+1,i}\subset Q_{k,j}\}}\int_{Q_{k+1,i}}
      f(x)x^\gamma\,dx\\
    &\quad-\int_{Q_{k,j}}P_{Q_{k,j}}^{(s)}(f)(x)x^\gamma\,dx\\
    &=\int_{Q_{k,j}}\lf[f(x)-P_{Q_{k,j}}^{(s)}(f)(x)\r]x^\gamma\,dx=0.
  \end{align*}
  This finishes the proof of Lemma \ref{atomdecomposition}.
\end{proof}

\begin{proof}[Proof of Proposition \ref{HKuvb=HKu8b}]
  Let $1<u<v\le\fz$, $s\in\zz_+$ and $\alpha\in[0,\infty)$. We first show
that
\begin{equation}\label{4.1x}
HK_{(u,\fz,s)_\alpha}(\mathcal{X})\subset
  HK_{(u,v,s)_\alpha}(\mathcal{X}).
\end{equation}
To this end, let
  $g\in HK_{(u,\fz,s)_\alpha}(\mathcal{X})$.
  Then, by Definition \ref{HKuvb}, we know that there exist $(u,\fz,s)_\alpha$-polymers $\{g_i\}_i$ such that
  $g=\sum_i g_i$ in $(JN_{(u',1,s)_\alpha}(\mathcal{X}))^\ast$ and
  $$\sum_i\|g_i\|_{\widetilde{HK}_{(u,\fz,s)_\alpha}(\mathcal{X})}
  \lesssim\|g\|_{HK_{(u,\fz,s)_\alpha}(\mathcal{X})}.$$
  By Proposition \ref{JNpqa=JNp1a}, we conclude that
  $(JN_{(u',1,s)_\alpha}(\mathcal{X}))^\ast=(JN_{(u',v',s)_\alpha}(\mathcal{X}))^\ast$,
  which implies that $g=\sum_i g_i$ in $(JN_{(u',v',s)_\alpha}(\mathcal{X}))^\ast$.
  Moreover, from Definition \ref{uvas-polymer}, we deduce that, for any $i$,
  $$\|g_i\|_{\widetilde{HK}_{(u,v,s)_\alpha}(\mathcal{X})}
  \le\|g_i\|_{\widetilde{HK}_{(u,\fz,s)_\alpha}(\mathcal{X})}$$
  and hence
  $$\|g\|_{HK_{(u,v,s)_\alpha}(\mathcal{X})}
  \le \sum_i\|g_i\|_{\widetilde{HK}_{(u,v,s)_\alpha}(\mathcal{X})}
  \le \sum_i\|g_i\|_{\widetilde{HK}_{(u,\fz,s)_\alpha}(\mathcal{X})}
  \lesssim\|g\|_{HK_{(u,\fz,s)_\alpha}(\mathcal{X})},$$
which further implies \eqref{4.1x} holds true.

  To show the converse of \eqref{4.1x}, namely,
$HK_{(u,v,s)_\alpha}(\mathcal{X})\subset HK_{(u,\fz,s)_\alpha}(\mathcal{X})$,
  by $(JN_{(u',v',s)_\alpha}(\mathcal{X}))^\ast
  =(JN_{(u',1,s)_\alpha}(\mathcal{X}))^\ast$ again, it suffices to prove that
every $(u,v,s)_\alpha$-polymer $g$ can be decomposed into
  $g=\sum_{k=-1}^\infty g_k$ in $(JN_{(u',1,s)_\alpha}(\mathcal{X}))^\ast$,
  where $\{g_k\}_{k=-1}^\infty$ are $(u,\infty,s)_\alpha$-polymers, and
\begin{equation}\label{4.1y}
\sum_{k=-1}^\infty \|g_k\|_{\widetilde{HK}_{(u,\infty,s)_\alpha}(\mathcal{X})}
\lesssim\|g\|_{\widetilde{HK}_{(u,v,s)_\alpha}(\mathcal{X})}.
\end{equation}

Indeed, for any $(u,v,s)_\alpha$-polymer $g$, by Definition \ref{HKuvb},
we know that there exist $(u,v,s)_\alpha$-atoms
  $\{A_\ell\}_\ell$ on disjoint cubes $\{Q_\ell\}_\ell$ and
  $\{\lambda_\ell\}_\ell\subset\cc$
  such that $g=\sum_\ell \lambda_\ell A_\ell$ pointwise and
  \begin{align}\label{Aell<g}
    \lf(\sum_\ell |\lambda_\ell|^u\r)^\frac1u
    \lesssim\|g\|_{\widetilde{HK}_{(u,v,s)_\alpha}(\mathcal{X})}.
  \end{align}
For each $\ell$, applying Lemma \ref{atomdecomposition} with $f$ replaced by
  $A_\ell\in L_s^v(Q_\ell)$ and
  $\lambda:=\lf(\fint_{Q_\ell}|A_\ell|^v\r)^{\frac 1v}$,
we then obtain
$A_\ell-P_{Q_\ell}^{(s)}(A_\ell)=\sum_{k=0}^\infty\sum_j a_{k,j}^{(\ell)}$
  almost everywhere, where, for any $k\in\zz_+$ and any $j$,
$\supp(a_{k,j}^{(\ell)})\subset Q_{k,j}^{(\ell)}$,
$a_{k,j}^{(\ell)}\in L_s^\infty(Q_{k,j}^{(\ell)})$,
\begin{align}\label{akjl}
    \lf\|a_{k,j}^{(\ell)}\r\|_{L^\infty(Q_{k,j}^{(\ell)})}
    \lesssim C_1^k\lf(\fint_{Q_\ell}|A_\ell|^v\r)^{\frac 1v},
  \end{align}
$\{Q_{k,j}^{(\ell)}\}_j\subset Q_\ell$ are disjoint and
  \begin{align}\label{Qkjl}
    \bigcup_j Q_{0,j}^{(\ell)}=Q_\ell,\quad
    \bigcup_j Q_{k,j}^{(\ell)}
    =\lf\{x\in Q_\ell:\,\,\mathcal{M}_{Q_\ell}^{(d)}A_\ell(x)>
     C_1^k\lf(\fint_{Q_\ell}|A_\ell|^v\r)^{\frac 1v}\r\},\quad\forall\,k\in\nn.
  \end{align}
Since $\{Q_\ell\}_\ell$ are disjoint, it follows that
  $\{Q_{k,j}^{(\ell)}\}_{\ell,j}$ are disjoint for any $k\in\zz_+$.
  Therefore, we can define
  $$g_{-1}:=\sum_\ell \lambda_\ell P_{Q_\ell}^{(s)}(A_\ell)
  \quad\mathrm{and}\quad
  g_k:=\sum_{\ell,j}\lambda_\ell a_{k,j}^{(\ell)},
  \quad\forall\,k\in\zz_+$$
  almost everywhere, which implies that $g=\sum_{k=-1}^\fz g_k$ almost everywhere.
  It remains to show
  \begin{enumerate}
    \item[{\rm(i)}]For any $k\in\{-1,0,1,\dots\}$, $g_k$ is a $(u,\infty,s)_\alpha$-polymer;
    \item[{\rm(ii)}]$\sum_{k=-1}^\infty\|g_k\|_{\widetilde{HK}_{(u,\fz,s)_\alpha}(\mathcal{X})}
      \lesssim\|g\|_{\widetilde{HK}_{(u,v,s)_\alpha}(\mathcal{X})}$;
    \item[{\rm(iii)}]$g=\sum_{k=-1}^\infty g_k$ in $(JN_{(u',1,s)_\alpha}(\mathcal{X}))^\ast$.
  \end{enumerate}

  We first show (i).
  When $k=-1$, for any $\ell$, by $A_\ell\in L_s^v(Q_\ell)$, \eqref{P_Q^s(f)} and
  \eqref{C_{(s)}}, we have $P_{Q_\ell}^{(s)}(A_\ell)\in L_s^\infty(Q_\ell)$
  and $\|P_{Q_\ell}^{(s)}(A_\ell)\|_{L^\infty(Q_\ell)}
  \ls \fint_{Q_\ell}|A_\ell|$.
  Let $$\widetilde{A_\ell}:=\frac{|Q_\ell|^{-\frac1u-\alpha}}
  {\|P_{Q_\ell}^{(s)}(A_\ell)\|_{L^\infty(Q_\ell)}}P_{Q_\ell}^{(s)}(A_\ell).$$
  Then $\{\widetilde{A_\ell}\}_\ell$ are $(u,\fz,s)_\alpha$-atoms and
  $g_{-1}=\sum_\ell \lambda_\ell |Q_\ell|^{\frac1u+\alpha}
   \|P_{Q_\ell}^{(s)}(A_\ell)\|_{L^\infty(Q_\ell)}\widetilde{A_\ell}$.
  From this, the H\"{o}lder inequality,
  Definition \ref{uvas-atom}(ii) for $(u,v,s)_\alpha$-atoms $\{A_\ell\}_\ell$
  and \eqref{Aell<g}, we deduce that
  \begin{align}\label{g_-1}
    \|g_{-1}\|_{\widetilde{HK}(u,\infty,s)_\alpha(\mathcal{X})}
    &\le \lf[\sum_\ell |\lambda_\ell|^u|Q_\ell|^{1+u\alpha}\lf\|P_{Q_\ell}^{(s)}(A_\ell)\r\|
          _{L^\infty(Q_\ell)}^u\r]^\frac1u\\
    &\ls \lf[\sum_\ell |\lambda_\ell|^u |Q_\ell|^{1+u\alpha}
           \lf(\fint_{Q_\ell}|A_\ell|\r)^u\r]^\frac1u\notag\\
    &\ls \lf[\sum_\ell |\lambda_\ell|^u |Q_\ell|^{1+u\alpha}
           \lf(\fint_{Q_\ell}|A_\ell|^v\r)^\frac uv\r]^\frac1u\noz\\
    &\ls \lf[\sum_\ell |\lambda_\ell|^u |Q_\ell|^{1+u\alpha-\frac uv}
           \lf\|A_\ell \r\|_{L^v(Q_\ell)}^u\r]^\frac1u\notag\\
    &\ls \lf(\sum_\ell |\lambda_\ell|^u\r)^\frac1u
      \lesssim\|g\|_{\widetilde{HK}_{(u,v,s)_\alpha}(\mathcal{X})}\notag.
  \end{align}
  By \eqref{g_-1} and the disjointness of $\{Q_\ell\}_\ell$, we conclude that
  $g_{-1}$ is a $(u,\infty,s)_\alpha$-polymer.

  When $k=0$, since $Q_{0,0}^{(\ell)}=Q_\ell$ and $Q_{0,j}^{(\ell)}=\emptyset$
  for any $j>0$ as in Lemma \ref{atomdecomposition}, then
  $g_0=\sum_{\ell,j}\lambda_\ell a_{0,j}^{(\ell)}=\sum_\ell \lambda_\ell a_{0,0}^{(\ell)}$.
  By \eqref{akjl}, we have $\|a_{0,0}^{(\ell)}\|_{L^\infty(Q_\ell)}
    \lesssim (\fint_{Q_\ell}|A_\ell|^v)^{\frac 1v}$.
  Let $$\widetilde{a_{0,0}^{(\ell)}}:=\frac{|Q_\ell|^{-\frac1u-\alpha}}
  {\|a_{0,0}^{(\ell)}\|_{L^\infty(Q_\ell)}}a_{0,0}^{(\ell)}.$$
  Then $\{\widetilde{a_{0,0}^{(\ell)}}\}_\ell$ are $(u,\fz,s)_\alpha$-atoms and
  $g_0=\sum_\ell \lambda_\ell |Q_\ell|^{\frac1u+\alpha}
   \|a_{0,0}^{(\ell)}\|_{L^\infty(Q_\ell)} \widetilde{a_{0,0}^{(\ell)}}$.
  From this,  \eqref{akjl},
  Definition \ref{uvas-atom}(ii) for $(u,v,s)_\alpha$-atoms $\{A_\ell\}_\ell$
  and \eqref{Aell<g}, it follows that
  \begin{align}\label{g_0}
    \lf\|g_0\r\|_{\widetilde{HK}_{(u,\fz,s)_\alpha}(\mathcal{X})}
    &\le\lf[\sum_{\ell} |\lambda_\ell|^u \lf|Q_\ell\r|^{1+u\alpha}
      \lf\|a_{0,0}^{(\ell)}\r\|_{L^\infty(Q_\ell)}^u\r]^\frac1u
      \lesssim\lf[\sum_{\ell} |\lambda_\ell|^u \lf|Q_\ell\r|^{1+u\alpha}
      \lf(\fint_{Q_\ell}|A_\ell|^v\r)^{\frac uv}\r]^\frac1u \\
    &\sim \lf[\sum_{\ell} |\lambda_\ell|^u \lf|Q_\ell\r|^{1+u\alpha-\frac uv}
      \|A_\ell\|_{L^v(Q_\ell)}^u\r]^\frac1u
      \lesssim \lf(\sum_\ell |\lambda_\ell|^u\r)^\frac1u
      \lesssim \|g\|_{\widetilde{HK}_{(u,v,s)_\alpha}(\mathcal{X})}. \notag
  \end{align}
  By \eqref{g_0} and the disjointness of $\{Q_{0,j}^{(\ell)}\}_{\ell,j}$,
  we conclude that $g_{0}$ is a $(u,\infty,s)_\alpha$-polymer.

  When $k\in\nn$, for any $\ell$ and $j$,
  let $$\widetilde{a_{k,j}^{(\ell)}}:=\frac{|Q_{k,j}^{(\ell)}|^{-\frac1u-\alpha}}
  {\|a_{k,j}^{(\ell)}\|_{L^\infty(Q_{k,j}^{(\ell)})}}a_{k,j}^{(\ell)}.$$
  Then $\{\widetilde{a_{k,j}^{(\ell)}}\}_{\ell,j}$ are $(u,\fz,s)_\alpha$-atoms and
  $$g_k=\sum_{\ell,j}\lambda_\ell |Q_{k,j}^{(\ell)}|^{\frac1u+\alpha}
  \|a_{k,j}^{(\ell)}\|_{L^\infty(Q_{k,j}^{(\ell)})} \widetilde{a_{k,j}^{(\ell)}}.$$
  From this, \eqref{akjl}, \eqref{Qkjl}, $u\in(1,\fz)$, $\alpha\in[0,\infty)$,
  Definition \ref{uvas-atom}(ii) for $(u,v,s)_\alpha$-atoms
  $\{A_\ell\}_\ell$ and \eqref{Aell<g}, it follows that
  \begin{align}\label{gk polymer}
    \lf\|g_k\r\|_{\widetilde{HK}_{(u,\fz,s)_\alpha}(\mathcal{X})}
    &\le\lf[\sum_{\ell,j} |\lambda_\ell|^u \lf|Q_{k,j}^{(\ell)}\r|^{1+u\alpha}
      \lf\|a_{k,j}^{(\ell)}\r\|_{L^\infty(Q_{k,j}^{(\ell)})}^u\r]^\frac1u\\
    &\lesssim C_1^k\lf[\sum_{\ell,j} |\lambda_\ell|^u \lf|Q_{k,j}^{(\ell)}\r|^{1+u\alpha}
      \lf(\fint_{Q_\ell}|A_\ell|^v\r)^{\frac uv}\r]^\frac1u \notag\\
    &\lesssim C_1^k\lf[\sum_\ell |\lambda_\ell|^u \lf|
      \lf\{x\in Q_\ell:\,\,\mathcal{M}_{Q_\ell}^{(d)}A_\ell(x)>
      C_1^k\lf(\fint_{Q_\ell}|A_\ell|^v\r)^{\frac 1v}\r\}\r|^{1+u\alpha}\r.\notag\\
    &\lf.\qquad\qquad\times\lf(\fint_{Q_\ell}|A_\ell|^v\r)^{\frac uv}\r]^\frac1u\notag\\
    &\lesssim C_1^k\lf[\sum_{\ell} |\lambda_\ell|^u \lf|Q_\ell\r|^{1+u\alpha}
      \lf(\fint_{Q_\ell}|A_\ell|^v\r)^{\frac uv}\r]^\frac1u\notag\\
    &\lesssim C_1^k\lf(\sum_\ell |\lambda_\ell|^u\r)^\frac1u
      \lesssim C_1^k\|g\|_{\widetilde{HK}_{(u,v,s)_\alpha}(\mathcal{X})}. \notag
  \end{align}
  By \eqref{gk polymer} and the disjointness of $\{Q_{k,j}^{(\ell)}\}_{\ell,j}$,
  we conclude that $g_{k}$ is a $(u,\infty,s)_\alpha$-polymer, which completes
the proof of (i).

  Next, we show (ii).
  From \eqref{g_-1}, \eqref{g_0} and the proof of \eqref{gk polymer}, it follows that
  \begin{align*}
    &\sum_{k=-1}^\infty\lf\|g_k\r\|_{\widetilde{HK}_{(u,\fz,s)_\alpha}(\mathcal{X})}\\
    &\quad=\lf\|g_{-1}\r\|_{\widetilde{HK}_{(u,\fz,s)_\alpha}(\mathcal{X})}+
           \lf\|g_0\r\|_{\widetilde{HK}_{(u,\fz,s)_\alpha}(\mathcal{X})}+
           \sum_{k=1}^\fz\lf\|g_k\r\|_{\widetilde{HK}_{(u,\fz,s)_\alpha}(\mathcal{X})}
     \lesssim \|g\|_{\widetilde{HK}_{(u,v,s)_\alpha}(\mathcal{X})}   \\
    &\quad\quad+
      \sum_{k=1}^\infty C_1^k\lf[\sum_\ell |\lambda_\ell|^u \lf|
      \lf\{x\in Q_\ell:\,\,\mathcal{M}_{Q_\ell}^{(d)}A_\ell(x)>
      C_1^k\lf(\fint_{Q_\ell}|A_\ell|^v\r)^{\frac 1v}\r\}\r|^{1+u\alpha}
      \lf(\fint_{Q_\ell}|A_\ell|^v\r)^{\frac uv}\r]^\frac1u\\
    &\quad=:\|g\|_{\widetilde{HK}_{(u,v,s)_\alpha}(\mathcal{X})}+\mathrm{II}.
  \end{align*}
  By $u<v$, we can write $C_1^k=C_1^{-k\epsilon}C_1^{k \frac vu}$ with
  $\epsilon:=\frac vu-1>0$. From this, the H\"{o}lder inequality,
  $u\in(1,\fz)$, $\alpha\in[0,\infty)$, Lemma \ref{Cavalieri} with
  $\delta_k:=C_1^k\lf(\fint_{Q_\ell}|A_\ell|^v\r)^{1/v}$ [it is easy to see that
  $\{\delta_k\}_{k\in\nn}$ satisfy (i) and (ii) of Lemma \ref{Cavalieri} because
  $C_1\in(2^n,\fz)$],
  the boundedness of $\mathcal{M}_{Q_\ell}^{(d)}$ on $L^{v}(Q_\ell)$,
  Definition \ref{uvas-atom}(ii) for $(u,v,s)_\alpha$-atoms $\{A_\ell\}_\ell$
  and \eqref{Aell<g}, we deduce that
  \begin{align*}
    \mathrm{II}&=\sum_{k=1}^\infty C_1^{-k\epsilon}C_1^{k \frac vu}\lf[\sum_\ell
      |\lambda_\ell|^u \lf|
      \lf\{x\in Q_\ell:\,\,\mathcal{M}_{Q_\ell}^{(d)}A_\ell(x)>
      \delta_k\r\}\r|^{1+u\alpha}
      \lf(\fint_{Q_\ell}|A_\ell|^v\r)^{\frac uv}\r]^\frac1u\\
    &\le\lf(\sum_{k=1}^\infty C_1^{-k\epsilon u'}\r)^\frac{1}{u'}
      \lf[\sum_{k=1}^\infty C_1^{kv}\sum_\ell |\lambda_\ell|^u \lf|
      \lf\{x\in Q_\ell:\,\,\mathcal{M}_{Q_\ell}^{(d)}A_\ell(x)>
      \delta_k\r\}\r|^{1+u\alpha}
      \lf(\fint_{Q_\ell}|A_\ell|^v\r)^{\frac uv}\r]^\frac1u\\
    &\sim\lf[\sum_\ell |\lambda_\ell|^u \lf(\fint_{Q_\ell}|A_\ell|^v\r)^\frac{u-v}{v}
      \sum_{k=1}^\infty \delta_k^v
      \lf|\lf\{x\in Q_\ell:\,\,\mathcal{M}_{Q_\ell}^{(d)}A_\ell(x)>
      \delta_k\r\}\r|^{1+u\alpha}\r]^\frac1u\\
    &\lesssim\lf[\sum_\ell |\lambda_\ell|^u
      \lf(\fint_{Q_\ell}|A_\ell|^v\r)^\frac{u-v}{v}|Q_\ell|^{u\alpha}
      \sum_{k=1}^\infty \delta_k^v
      \lf|\lf\{x\in Q_\ell:\,\,\mathcal{M}_{Q_\ell}^{(d)}A_\ell(x)>
      \delta_k\r\}\r|\r]^\frac1u\\
    &\lesssim\lf[\sum_\ell |\lambda_\ell|^u
      \lf(\fint_{Q_\ell}|A_\ell|^v\r)^\frac{u-v}{v}|Q_\ell|^{u\alpha}
      \lf\|\mathcal{M}_{Q_\ell}^{(d)}A_\ell\r\|_{L^{v}(Q_\ell)}^v\r]^\frac1u\\
    &\lesssim\lf[\sum_\ell |\lambda_\ell|^u \lf(\fint_{Q_\ell}|A_\ell|^v\r)^\frac{u-v}{v}
      |Q_\ell|^{u\alpha}\lf\|A_\ell\r\|_{L^{v}(Q_\ell)}^v\r]^\frac1u\\
    &\sim\lf\{\sum_\ell |\lambda_\ell|^u |Q_\ell|^{u\alpha-\frac{u-v}{v}}
      \lf\|A_\ell\r\|_{L^{v}(Q_\ell)}^u\r\}^\frac1u
      \lesssim \lf\{\sum_\ell |\lambda_\ell|^u |Q_\ell|^{u\alpha-\frac{u-v}{v}}
      |Q_\ell|^{u(\frac1v-\frac1u-\alpha)}\r\}^\frac1u\\
    &\sim \lf(\sum_\ell |\lambda_\ell|^u\r)^\frac1u
    \lesssim\|g\|_{\widetilde{HK}_{(u,v,s)_\alpha}(\mathcal{X})}.
  \end{align*}

  Finally, we show (iii).
  Indeed, for any $K\in\{-1,0,1,...\}$ and $f\in JN_{(u',1,s)_\alpha}(\mathcal{X})$,
  by Proposition \ref{weak-star}, we conclude that
  \begin{align*}
    \lf|\int_\mathcal{X}\sum_{k=K}^\fz g_k f \r|
    \le \sum_{k=K}^\fz \lf\|g_k \r\|_{\widetilde{HK}_{(u,\fz,s)_\alpha}(\mathcal{X})}
        \lf\|f \r\|_{JN_{(u',1,s)_\alpha}(\mathcal{X})}.
  \end{align*}
  From this and $\sum_{k=-1}^\infty\|g_k\|_{\widetilde{HK}_{(u,\fz,s)_\alpha}(\mathcal{X})}
  \lesssim\|g\|_{\widetilde{HK}_{(u,v,s)_\alpha}(\mathcal{X})}$, it follows that,
  for any $f\in JN_{(u',1,s)_\alpha}(\mathcal{X})$,
  $$\sum_{k=-1}^\fz \int_\mathcal{X}g_k f= \int_\mathcal{X} \sum_{k=-1}^\fz g_k f$$
  and hence
  $$\langle g,f\rangle=\lf\langle\sum_{k=-1}^\fz g_k,f\r\rangle
  =\int_\mathcal{X} \sum_{k=-1}^\fz g_k f
  =\sum_{k=-1}^\fz \int_\mathcal{X}g_k f=\sum_{k=-1}^\fz \langle g_k,f\rangle,$$
  which implies that $g=\sum_{k=-1}^\infty g_k$ in $(JN_{(u',1,s)_\alpha}(\mathcal{X}))^\ast$.

  This finishes the proof of Proposition \ref{HKuvb=HKu8b}.
\end{proof}

\noindent\textbf{Acknowledgements}. Dachun Yang and Wen Yuan would like to thank Professors
Hong Yue and Juha Kinnunen, and Jin Tao would like to thank Ziyi He and Professor Dongyong Yang
for some useful discussions on the subject of this article.

\bigskip

\noindent Jin Tao, Dachun Yang  and Wen Yuan (Corresponding author)

\medskip

\noindent Laboratory of Mathematics and Complex Systems
(Ministry of Education of China),
School of Mathematical Sciences, Beijing Normal University,
Beijing 100875, People's Republic of China

\smallskip

\noindent{\it E-mails:} \texttt{jintao@mail.bnu.edu.cn} (J. Tao)

\noindent\phantom{{\it E-mails:}} \texttt{dcyang@bnu.edu.cn} (D. Yang)

\noindent\phantom{{\it E-mails:}} \texttt{wenyuan@bnu.edu.cn} (W. Yuan)

\end{document}